\documentclass[12 pt]{amsart}
\usepackage{amscd,amssymb,amsmath,amsthm}
\usepackage{tikz}
\usetikzlibrary{decorations.pathmorphing}
\usepackage{amssymb,amscd,latexsym,epsfig,xypic,hyperref}
\usepackage[mathscr]{euscript}
\usepackage{eufrak}
\newcommand{\Y}{{\mathcal{Y}}}

\newcommand{\proj}{\mathbb{P}}
\newcommand{\Z}{\mathbb{Z}}
\newcommand{\ZZ}{{\mathsf{Z}}}

\newcommand{\OO}{\mathcal{O}}
\newcommand{\com}{\mathbb{C}}

\newcommand{\MM}{{\mathcal{M}}}

\newcommand{\cal}{\mathcal}

\DeclareFontFamily{OT1}{rsfs}{}
\DeclareFontShape{OT1}{rsfs}{n}{it}{<-> rsfs10}{}
\DeclareMathAlphabet{\curly}{OT1}{rsfs}{n}{it}

\newcommand\y{{\tilde{y}}}

\renewcommand\O{\mathcal O}
\newcommand\F{\mathcal F}

\newcommand\LL{\mathbb L}
\newcommand\PP{\mathbb P}

\newcommand\C{\mathbb C}

\newcommand{\bZ}{{\mathbb{Z}}}
\newcommand{\bL}{{\mathsf{L}}}
\newcommand{\bC}{{\mathbb{C}}}
\newcommand{\bA}{{\mathbb{A}}}
\newcommand{\bP}{{\mathbb{P}}}

\newcommand{\rt}[1]{\stackrel{#1\,}{\rightarrow}}
\newcommand{\Rt}[1]{\stackrel{#1\,}{\longrightarrow}}
\newcommand\To{\longrightarrow}
\newcommand\into{\hookrightarrow}
\newcommand\ito{\ar@{^{ (}->}[r]}
\newcommand{\Into}{\ensuremath{\lhook\joinrel\relbar\joinrel\rightarrow}}

\newfont{\bigtimesfont}{cmsy10 scaled \magstep5}
\newcommand{\bigtimes}{\mathop{\lower0.9ex\hbox{\bigtimesfont\symbol2}}}

\newcommand\id{\operatorname{id}}

\newcommand\Hom{\operatorname{Hom}}
\renewcommand\hom{\curly H\!om}

\newcommand\Pic{\operatorname{Pic}}

\newcommand\Hilb{\operatorname{Hilb}}

\newcommand\beq[1]{\begin{equation}\label{#1}}
\newcommand\eeq{\end{equation}}
\newcommand\beqa{\begin{eqnarray*}}
\newcommand\eeqa{\end{eqnarray*}}

\makeatletter \@addtoreset{equation}{section} \makeatother

\newtheorem{lem}[equation]{Lemma}

\newtheorem{prop}[equation]{Proposition}

\begin{document}

\title{On the motivic stable pairs invariants of $K3$ surfaces}
\author{S. Katz, A. Klemm, and R. Pandharipande\\
with an Appendix by R. P. Thomas}
\date{June 2019}
\maketitle

\begin{abstract}
For a $K3$ surface $S$ and a class $\beta \in \text{Pic}(S)$,
we study motivic invariants of stable pairs moduli spaces
associated to  3-fold
thickenings of $S$.
We conjecture suitable deformation and divisibility invariances
for the Betti realization.
Our conjectures, together with earlier  calculations of Kawai-Yoshioka,
imply a full determination of the theory in terms
of the Hodge numbers of the Hilbert schemes of points of $S$.
The work may be viewed as the third  in a sequence of formulas
starting with Yau-Zaslow and Katz-Klemm-Vafa (each recovering the
former). Numerical data suggest the
motivic invariants are linked to the Mathieu $\mathsf{M}_{24}$ moonshine phenomena.

The KKV formula and the Pairs/Noether-Lefschetz
correspondence together determine the BPS counts
of $K3$-fibered Calabi-Yau 3-folds in fiber classes in
terms of modular forms.
We propose a framework for a refined P/NL correspondence
for the motivic invariants of $K3$-fibered CY 3-folds. 
For the STU model, a complete conjecture is
presented. 
\end{abstract}

\setcounter{tocdepth}{1} 
\tableofcontents

\setcounter{section}{-1}

\baselineskip=18pt

\section{Introduction} \label{intro}
A beautiful connection between curve counting on $K3$ surfaces
and modular forms was conjectured in 1995 by Yau and Zaslow \cite{YZ}:
the generating series of the counts of rational curves in
primitive classes was conjectured to equal the inverse of the  discriminant 
$$\Delta(q)= q \prod_{n=1}^\infty (1-q^n)^{24}\ .$$ 
By  work of G\"ottsche \cite{G}, $\Delta(q)^{-1}$
was already known to 
arise as the generating series of the Euler characteristics of
Hilbert schemes of points of $K3$ surfaces $S$,
$$\sum_{n\geq 0} \chi\big(\text{Hilb}^n(S)\big) q^{n-1} = \frac{1}{\Delta(q)}\ .$$
An argument by Beauville \cite{B} in 1997 provided a geometric link
between curve counting in primitive classes and
the Euler characteristics of $\text{Hilb}^n(S)$.

A connection between the higher genus curve counts
on $K3$ surfaces and the generating series of $\chi_y$ genera
of $\text{Hilb}^n(S)$,
%AK: here i added a label
\begin{multline*}
%\label{unrefined}   
\sum_{n\geq 0} \chi_y\big(\text{Hilb}^n(S)\big) q^{n-1} = \frac{1}{q\prod_{n=1}^{\infty}
(1-yq^n)^2(1-q^n)^{20}(1-y^{-1}q^n)^2}\ ,
\end{multline*}
was proposed in 1999~\cite{KKV}. 
The conjectures of \cite{KKV} govern {\em all}
classes on $K3$ surfaces via a subtle divisibility invariance for 
 multiple classes. In the genus 0 primitive
case, the Yau-Zaslow conjecture is recovered. 
A proof of the KKV conjecture was recently found \cite{PT2}.
The moduli of sheaves (via stable pairs \cite{PT1}) on $K3$
surfaces play a central role. 

We propose here a third step in the sequence of conjectures
starting with Yau-Zaslow and KKV. We conjecture the
Betti realization of the motivic stable pairs theory of 
$K3$ surfaces is connected to the generating series
of Hodge numbers 
of $\text{Hilb}^n(S)$,
\begin{multline*}
\sum_{n\geq 0} \chi_{\text{Hodge}}\big(
\text{Hilb}^n(S)\big) q^{n-1} = \\
\frac{1}{q
\prod_{n=1}^\infty
(1-u^{-1}y^{-1} q^n)(1-u^{-1}y q^n) (1-q^n)^{20}(1-uy^{-1}q^n)
(1-u yq^n)}\ ,
\end{multline*}
where the variables $u$ and $y$ keep track of the Hodge grading. 
Our conjecture governs all curve classes and specializes
to the KKV conjecture after taking Euler characteristics.

In addition to the surprising divisibility invariance
already present in the KKV conjecture \cite{KKV}, we propose 
a new deformation
invariance of the Betti realization of the motivic invariants
of $K3$ geometries. To support our conjectures, we
provide a few basic calculations.

The KKV formula and the Pairs/Noether-Lefschetz
correspondence \cite{gwnl,PT2} together determine the BPS counts
of $K3$-fibered Calabi-Yau 3-folds in fiber classes in
terms of modular forms. 
We propose a framework for a refined P/NL correspondence
for the motivic invariants of $K3$-fibered Calabi-Yau 3-folds.
For the STU model, a complete conjecture is
provided. 

In the Appendix by R. Thomas, the Gopakumar-Vafa
perspective on the motivic invariants of $K3$ surfaces
is discussed. The relevant moduli spaces \cite{SKatz} are shown to
be nonsingular
even in the imprimitive case.
The results of the Appendix may be viewed as supporting our motivic stable
pairs conjectures in the larger framework of the
conjectural Pairs/Gopakumar-Vafa correspondence.{\footnote{An
elementary overview of the various correspondences for
curve counts on Calabi-Yau 3-folds can be found in \cite{13/2}.}}

\vspace{15pt}
\noindent{\em Acknowledgements}.
The formulation of the conjectures (and especially of
the Betti deformation invariance) was undertaken while
 A.K. and R.P. were visiting S.K. at the University of Illinois
in April 2014. 

We thank J. Bryan, J. Choi, D. Maulik, E. Scheidegger, and
V. Shende for many conversations
directly related to the motivic stable pairs invariants of $K3$ surfaces.
Much of the paper was written while R.P. was attending
the conference
{\em K3 surfaces and their moduli} on the island of 
Schiermonnikoog in May 2014 (organized
by C. Faber, G. Farkas, and G. van der Geer with support
from the {\em Compositio foundation}). Discussions there with B. Bakker, A. Bayer, D. Huybrechts, 
E. Macri, G. Oberdieck, and Q. Yin were very helpful. 
We thank G. Moore for suggesting a connection to Mathieu moonshine at
{\em String Math 2014} in Edmonton and M. Gaberdiel for related conversations.

Special thanks are due to R. Thomas for crucial help with the superpotential
investigations and for contributing the Appendix on the Gopakumar-Vafa
moduli approach.

S.K. was supported by NSF grant DMS-12-01089. A.K. 
was  supported by KL 2271/1-1 and 
DMS-11-59265.
R.P. was supported by 
grants SNF-200021-143274 and ERC-2012-AdG-320368-MCSK.

\section{Curve classes on $K3$ surfaces} \label{cc}
Let $S$ be a nonsingular projective $K3$ surface.
 The second cohomology of $S$ is a rank 22 lattice
with intersection form 
\begin{equation}\label{ccet}
H^2(S,\mathbb{Z}) \stackrel{\sim}{=} U\oplus U \oplus U \oplus E_8(-1) \oplus E_8(-1)\,,
\end{equation}
where
$$U
= \left( \begin{array}{cc}
0 & 1 \\
1 & 0 \end{array} \right)$$
and 
$E_8(-1)$
is the (negative) Cartan matrix. The intersection form \eqref{ccet}
is even.

The {\em divisibility}\, $m_\beta$ is
the maximal positive integer dividing the lattice
element $\beta\in H^2(S,\mathbb{Z})$.
If the divisibility is 1,
$\beta$ is {\em primitive}.
Elements with
equal divisibility and norm square 
$\langle \beta, \beta \rangle$
are equivalent up to orthogonal transformation \cite{CTC}.

The {\em Picard} lattice of $S$ is the intersection
$$\text{Pic}(S) = H^2(S,\mathbb{Z}) \cap H^{1,1}(S,\com)\ .$$
For a family of nonsingular $K3$ surfaces 
$$\pi: \mathcal{X} \rightarrow (\Delta,0)$$ with special fiber
$\mathcal{X}_0\cong S$
and trivial local system
$R^2\pi_*\mathbb{Z}$, the Noether-Lefschetz locus
associated to $\gamma \in H^2(S,\mathbb{Z})$
is 
$$NL_\gamma = \left\{ p\in \Delta  \ | 
\ \gamma\in \text{Pic}(\mathcal{X}_p)\ \right\}\ .$$
The Noether-Lefschetz locus is naturally a subscheme  
$NL_\gamma\subset \Delta$.

\section{Stable pairs motivic invariants} \label{spm}

Let $S$ be a nonsingular projective $K3$ surface.
Curve counting on $S$ may be approached via the {\em reduced}
virtual fundamental class of the moduli space of
stable maps to $S$ or the stable pairs theory of Calabi-Yau
3-fold thickenings of $S$. An equivalence relating these
two counts is proven in \cite{PT2} essentially using
\cite{gwnl,PP}.
We are interested here in the motivic invariants associated
to $S$. 
Since no motivic theory is available on the Gromov-Witten
side, we will consider here the moduli spaces
of stable pairs.

BPS counts for $S$ via stable pairs
were defined in \cite{PT2}. The construction uses $K3$-fibrations
sufficiently transverse to Noether-Lefschetz loci.
We follow the geometric perspective of \cite{PT2} to define a motivic theory
associated to $S$.

Let 
 $\alpha \in \text{Pic}(S)$
be a nonzero class which is both positive (with respect to
any ample polarization of $S$) and primitive.
%Let $\alpha\in \text{Pic}(S)$ be a nonzero class
%which is both effective and primitive.
%Of course condition $(\star)$ is satisfied if $\text{Pic}(S)\stackrel{\sim}{=}
%\mathbb{Z}\alpha$.
Let $T$ be a nonsingular 3-dimensional quasi-projective
 variety,
$$\epsilon: T \rightarrow (\Delta,0)\,,$$
fibered in $K3$ surfaces over a pointed
curve $(\Delta,0)$
satisfying:
\begin{enumerate}
\item[(i)] $\Delta$ is a nonsingular quasi-projective curve with trivial
canonical class,
\item[(ii)] $\epsilon$ is smooth, projective, and $T_0 \cong S$,
%\item[(iii)] the local 
%Noether-Lefschetz locus $\text{NL}(\alpha)\subset \Delta$ 
%corresponding to
%the class $\alpha \in \text{Pic}(S)$ is the
%{\em reduced} point $0\in \Delta$.
\end{enumerate}

The class $\alpha \in \text{Pic}(S)$ is {\em $m$-rigid}
with respect to the family $\epsilon$ if the
following further condition is satisfied:
\begin{enumerate}
\item[$(\star)$] for every
 {\em effective} decomposition{\footnote{An effective
decomposition requires all parts $\gamma_i$ to be effective
divisors.}} 
$$m\alpha=\sum_{i=1}^l \gamma_i
\in \text{Pic}(S)\,,$$
the  local Noether-Lefschetz locus $\text{NL}(\gamma_i)
\subset \Delta$ corresponding to
each  class $\gamma_i \in \text{Pic}(S)$ is the
{\em reduced} point $0\in \Delta$.
\end{enumerate}
Let $\text{Eff}({m}\alpha) \subset \text{Pic}(S)$
denote the subset of effective summands of 
$m\alpha$.  
The existence of $m$-rigid families is easy to see \cite[Section 6.2]{PT2}.

%Condition $(\star)$ implies (iii).

Assume $\alpha$ is ${m}$-rigid
with respect to the family $\epsilon$.
By property $(\star)$,
there is a compact, open, and closed component  
$$P_n^\star(T,\gamma) \subset P_n(T,\gamma)$$ 
parameterizing 
 stable pairs{\footnote{For any class $\gamma\in \text{Pic}(S)$,
we denote the push-forward to $H_2(T,\mathbb{Z})$ also
by $\gamma$.
Let $P_n(T,\gamma)$ be the moduli space of stable pairs
of Euler characteristic $n$ and class $\gamma \in H_2(T,\mathbb{Z})$.}}
supported set-theoretically over the point 
$0\in \Delta$ for {\em every} effective summand $\gamma\in 
\text{Eff}({m}\alpha)$.
We define
\begin{equation}\label{kri}
\mathsf{W}_{n,\gamma}^\star(T) \in \mathsf{K}^{\widehat{\mu}}_{\text{var}}[\mathsf{L}^{-1}]
\end{equation}
to be the motivic{\footnote{At the moment, 
$\mathsf{W}^\star_{n,\gamma}(T)$
is defined only after a choice of orientation data is
made. Our discussion implicitly assumes {\em either} that there
is a canonical choice {\em or} that the choice does not affect the
motivic class for our $K3$ geometry (or, at the very least,
does not affect the Poincar\'e polynomials of the motivic classes
here). Certainly  $P_n^\star(T,\gamma)$ is often 
simply connected. If $\pi_1$ is trivial, then the
orientation is unique. Perhaps $P_n^\star(T,\gamma)$ is always simply connected?
}}
stable pairs invariant associated to
the component $P_n^\star(T,\gamma)$ following Joyce and
collaborators \cite{BJM}.

The motivic invariant \eqref{kri} takes values in
the Grothendieck
ring 
 of varieties
  carrying actions of groups of $n^{th}$ roots of unity $\mathsf{K}_{\text{var}}^{\widehat{\mu}}$ extended by the inverse of the Tate class,
$$\mathsf{L}= [\mathbb{A}^1]\, .$$
The product in $\mathsf{K}^{\widehat{\mu}}_{\mathrm{var}}$ is {\em not} induced by the ordinary product of varieties, but rather defined
explicitly by 
motivic convolution with a Fermat curve \cite{DL2,L}.

Let $\rho$ denote the canonical action of
the finite group scheme $$\mu_2=\{\pm1\}$$ on itself.  
We obtain an element
\begin{equation*}
[\mu_2,\rho]\in \mathsf{K}^{\widehat{\mu}}_{\mathrm{var}}\ .
\end{equation*}
Straightforward calculation using the definition of the product or the 
motivic Thom-Sebastiani formula yields the relation
\begin{equation}\label{jq}
\left(1-[\mu_2,\rho]
\right)^2=\bL,
\end{equation}
\noindent so we define $\bL^{\frac{1}{2}}$ by
\begin{equation}
\bL^{\frac{1}{2}}=1-[\mu_2,\rho]\, ,
\label{l12}
\end{equation}
see \cite[Remark 19]{KS}.
The ring $\mathsf{K}^{\widehat{\mu}}_{\mathrm{var}}[
\bL^{-1}]$ therefore contains all powers of $\bL^{\pm 1/2}$.

\vspace{7pt}
\noindent{\bf Definition.} 
{\em Let $\alpha\in \text{\em Pic}(S)$ be a primitive, positive class.
Given a family
$\epsilon: T \rightarrow (\Delta,0)$ 
satisfying conditions (i), (ii), and $(\star)$ for $m\alpha$, let
 \begin{multline*}
%\label{fbbb}
\sum_{n\in \mathbb{Z}}
\mathsf{V}_{n,m\alpha}^\epsilon(S)\ q^n  = \\
\text{\em Coeff}_{v^{m\alpha}} \left[
\log\left(1 + \sum_{n\in \mathbb{Z}} \sum_{\gamma\in \text{\em Eff}({m}\alpha)} 
q^n v^{\gamma}\,
\mathsf{W}_{n,\gamma}^\star(T)\right) \right] \, .
\end{multline*}}
\vspace{7pt}

The motivic invariant $\mathsf{V}^\epsilon_{n,m\alpha}(S)$ is the main topic of the
paper. The superscript $\epsilon$ records the family
$$\epsilon:T \rightarrow (\Delta,0)$$
used in the definition.

For positive $\beta \in \text{Pic}(S)$, we may write $\beta = m\alpha$
where $\alpha \in \text{Pic}(S)$ is positive and primitive
and $m=m_\beta$ is the divisibility of $\beta$. Hence,
$$\mathsf{V}_{n,\beta}^\epsilon(S) =
\mathsf{V}_{n,m \alpha}^\epsilon(S) $$
is defined.

\label{abc}

\section{Conjectures A and B} \label{con1}

\subsection{Poincar\'e polynomial}
We formulate here  several conjectures and speculations
 concerning the motivic
stable pairs invariants $V^\epsilon_{n,\beta}(S)$ introduced in
Section \ref{abc}. 
Let
$$\mathsf{H}^\epsilon_{n,\beta}(S)\in \mathbb{Q}[u]$$
denote the virtual Poincar\'e
polynomial
of the motivic invariant 
$\mathsf{V}^\epsilon_{n,\beta}(S)$.

\vspace{15pt}
\noindent {\bf Conjecture A.} {\em The virtual
Poincar\'e polynomial
$\mathsf{H}^\epsilon_{n,\beta}(S)$ is independent of 
the family 
$$\epsilon: T \rightarrow (\Delta,0)$$
satisfying conditions (i), (ii), and ($\star$) for
$\alpha=\frac{1}{m_\beta} \beta$ and $m=m_\beta$.}

\vspace{15pt}

Assuming the validity of Conjecture A, we may drop
the $\epsilon$ superscript and write $\mathsf{H}_{n,\beta}(S)$
for the virtual Poincar\'e polynomial.

\vspace{15pt}
\noindent {\bf Conjecture B.} {\em The virtual
Poincar\'e polynomial
$\mathsf{H}_{n,\beta}(S)$ is invariant
under deformations of $S$ for which $\beta$
remains algebraic.}

\vspace{15pt}

The divisibility $m_\beta$ and the norm square
$$\langle \beta, \beta \rangle = 2h-2$$
are the only deformation invariants of the
pair $(S,\beta)$ with $\beta\in \text{Pic}(S)$.
Assuming the validity of Conjecture B,
we write 
$$\mathsf{H}_{n,\beta}(S)= \mathsf{H}_{n,m_\beta,h}\ ,$$
replacing $(S,\beta)$  by $m_\beta$ and $h$.

\subsection{Motives}
Conjecture A for the motivic invariant  
$\mathsf{V}^\epsilon_{n,\beta}(S)$ seems not unreasonable: {\em $\mathsf{V}^\epsilon_{n,\beta}(S)$ is independent of $\epsilon$}.
However, Conjecture B is certainly false 
with $\mathsf{V}^\epsilon_{n,\beta}(S)$ in place of $\mathsf{H}_{n,\beta}(S)$
since the
class 
$$[S]\in \mathsf{K}_{\text{var}}$$
often appears in $\mathsf{V}^\epsilon_{n,\beta}(S)$. Examples
of the latter phenomenon can be found  already for
elliptically fibered $K3$ surfaces in the fiber class.
The classes of $K3$ surfaces derived equivalent to $S$
also appear in the elliptically fibered cases.

We define a new quotient ring of varieties:
$$\mathsf{K}^{D}_{\text{var}}
 = \mathsf{K}^{\widehat{\mu}}_{\text{var}}\ /\ 
\mathsf{I}^D\ ,$$
where $\mathsf{I}^D$ is the ideal generated by all differences
$$[X]- [\widetilde{X}] \ \in \mathsf{K}^{\widehat{\mu}}_{\text{var}}$$
where $X$ and $\widetilde{X}$ are $K3$ surfaces which
are derived equivalent.{\footnote{Or perhaps $\mathsf{I}^D$ should be
the ideal generated by the differences of {\em all}
derived equivalent nonsingular projective Calabi-Yau 
varieties. Alternatively, 
Chow motives may be a more natural framework for the entire discussion.
}}
Then, we could hope the class 
$$\overline{\mathsf{V}}^\epsilon_{n,m\alpha}(S)\in
\mathsf{K}_{\text{var}}^D[\mathsf{L}^{-1}]\, ,$$
obtained from $\mathsf{V}_{n,m\alpha}^\epsilon(S)$,
is a universal polynomial in the motivic powers of the 
class of the underlying
$K3$ surface, 
$$\overline{[\text{Sym}^r S]}\in \mathsf{K}_{\text{var}}^D[\mathsf{L}^{-1}]
\ .$$
The coefficients of such a polynomial would lie in
$\mathbb{Q}[\mathsf{L},\mathsf{L}^{-1}]$. 

An interesting related question immediately arises. Let
$S$ be nonsingular projective $K3$ surface with a positive {\em irreducible}
class $\alpha \in \text{Pic}(S)$. The moduli space
$P_n(S,\alpha)$ is known to be nonsingular \cite{KY,PT22}. Is there
a formula for 
$$\overline{[P_n(S,\alpha)]}\in \mathsf{K}_{\text{var}}^D[\mathsf{L}^{-1}]
$$
as a universal polynomial in the motivic powers $\overline{[\text{Sym}^rS]}$?

While we have (modest) computational evidence for 
Conjectures A and B constraining $\mathsf{H}^\epsilon_{m,\beta}(S)$, the above 
hopes for
  $\mathsf{V}^\epsilon_{m,\beta}(S)$
are simply speculations.

\section{Calculation of Kawai-Yoshioka} \label{ky}

In order to formulate Conjecture C
which completely determines $\mathsf{H}_{n,m,h}$, we  first
review the calculation of Kawai-Yoshioka.

 Let $P_n(S,h)$ denote the moduli space 
of stable pairs on the $K3$ surface $S$ for a positive irreducible
class $\alpha$ satisfying
$$
2h-2= \langle \alpha, \alpha \rangle .
$$
%The cotangent bundle $\Omega_{P}$
%of
%the moduli space $P_n(S,h)$ is the obstruction
%bundle of the reduced theory. Since the
%dimension of $P_n(S,h)$ is $2h-2+n+1$,
%$$I_h(y) = \frac{1}{t}\sum_{n} (-1)^{2h-1+n} e(P_n(S,h))\  y^n.$$
The moduli space $P_n(S,h)$ is nonsingular of dimension
$2h-1+n$.
The Poincar\'e polynomial{\footnote{All the cohomology of $P_n(S,h)$ is even.}} of 
$P_n(S,h)$,
$$\widetilde{\mathsf{H}}(P_n(S,h))=\sum_{i=0}^{2h-1+n} \text{dim}_{\mathbb{Q}}\,
H^{2i}(P_n(S,h),\mathbb{Q}) \ u^{2i}\ \ \ 
 \in \mathbb{Z}[u]\, ,$$ 
has been calculated by Kawai-Yoshioka.
By Theorem 5.158 of \cite{KY},
\begin{multline*}
 {\left(u^2\y-1\right)\left(1-{\y}^{-1}\right)} \cdot
\sum_{h=0}^\infty \sum_{n=1-h}^\infty  
\widetilde{\mathsf{H}}(P_n(S,h))\  u^{-2h} \y^n q^h = \\
\prod_{n=1}^\infty
 \frac{1}{
(1-\y q^n)(1-u^2\y q^n) (1-q^n)^{20}
(1-\y^{-1}q^n)(1-(u^2\y)^{-1}q^n)} \,.
\end{multline*}

In order to fit our motivic conventions in Section \ref{con2}, we define
$${\mathsf{H}}(P_n(S,h))=u^{-2h+1-n}\, \widetilde{\mathsf{H}}(P_n(S,h))$$
and rewrite the 
Kawai-Yoshioka formula
as
\begin{multline*}
 u^{-1}{\left(u^2\y-1\right)\left(1-{\y}^{-1}\right)} \cdot
\sum_{h=0}^\infty \sum_{n=1-h}^\infty  
\mathsf{H}(P_n(S,h))\  u^{n} \y^n q^h = \\
\prod_{n=1}^\infty
 \frac{1}{
(1-\y q^n)(1-u^2\y q^n) (1-q^n)^{20}
(1-\y^{-1}q^n)(1-(u^2 \y)^{-1}q^n)} \,.
\end{multline*}
After the substitution $y=u\y$, we find
\begin{multline*} 
 {\left(uy-1\right)\left(u^{-1}-{y}^{-1}\right)} \cdot
\sum_{h=0}^\infty \sum_{n=1-h}^\infty  
\mathsf{H}(P_n(S,h))\  y^n q^h = \\
\prod_{n=1}^\infty
 \frac{1}{
(1-u^{-1}y^{-1} q^n)(1-u^{-1}y q^n) (1-q^n)^{20}(1-uy^{-1}q^n)
(1-u yq^n)} \,.
\end{multline*}
The right side of the above formula is the generating series
of Hodge polynomials of the Hilbert schemes of points of $S$.

%\noindent {\bf Conjecture C.} {\em The Poincar\'e
%polynomial
%$\mathsf{H}_{n,m,h}$ is independent of $m$.}
\vspace{10pt}

\section{Refined Gopakumar-Vafa invariants} \label{gv}
\subsection{Definition}
Refined Gopakumar-Vafa invariants were defined mathematically in \cite{CKK}
via refined stable pairs invariants.{\footnote{Other definition
has been proposed in \cite{HST,KL}.
Connecting the refined invariants for $K3$ surfaces $\mathsf{R}^h_{j_L,j_R}$ 
defined here to the geometry of \cite{HST,KL} is the topic
of the Appendix by R. Thomas.}}
Following \cite{CKK}, 
we will define refined
invariants
$$\mathsf{R}_{j_L,j_R}^h\in \mathbb{Z}$$ for $K3$ surfaces
for all $h\in \mathbb{Z}$ and 
all half-integers
$$j_L,j_R \in\ \frac{1}{2}\mathbb{Z}_{\geq 0} =\big\{0,\frac{1}{2},1,
\frac{3}{2}, \ldots\big\}\ .$$ 
The definition uses formula (8.1) of \cite{CKK} and
the generating series of
Hodge numbers of the Hilbert schemes of points
of $K3$ sufaces of Section \ref{ky}.
If $h<0$, the definition 
is simple:
 $$\mathsf{R}_{j_L,j_R}^{h<0}=0\ .$$

For the $h\geq 0$ case, we 
will use the following notation.
For $j \in \frac{1}{2}\mathbb{Z}_{\geq 0}$, we define
%{\footnote{Since
%$j$ maybe a half-integer, an implicit square root may occur.
%For example $$[\frac{1}{2}]_{u^2}=u + u^{-1}\ .$$ \vspace{-10pt}}}
$$[j]_{x} = x^{-2j} + x^{-2j+2} + \cdots +
x^{2j-2}+ x^{2j}\ .$$
We define the refined invariants
$\mathsf{R}_{j_L,j_R}^{h\geq 0}$ for $K3$ surfaces  by
\begin{multline}\label{DDD}
\sum_{h=0}^\infty \sum_{j_L}
\sum_{j_R} 
%(-1)^{2(j_L+j_R)}\mathsf{R}_{j_L,j_R}^h\ 
%[j_L]_{(-u)^2}[j_R]_{(-y)^2}
\mathsf{R}_{j_L,j_R}^h\ 
[j_L]_{y}[j_R]_{u}\,
%\frac{(-u)^{-2j_L} - (-u)^{2j_L+2}}{1-u^2}
%\frac{(-y)^{-2j_R} - (-y)^{2j_R+2}}{1-q^2}
q^h= \\
\prod_{n=1}^\infty
 \frac{1}{
(1-u^{-1}y^{-1} q^n)(1-u^{-1}y q^n) (1-q^n)^{20}(1-uy^{-1}q^n)
(1-u y q^n)} \,.
\end{multline}
Here, the sums over $j_L$ and $j_R$ are both taken over
$\frac{1}{2}\mathbb{Z}_{\geq 0}$.

The above definition of 
  $\mathsf{R}_{j_L,j_R}^h$ will be discussed further in
Section \ref{con2}.  Formula \eqref{DDD} will be shown there to be
consistent with 
the Kawai-Yoshioka calculation of Section \ref{ky} via
the definition of the refined invariants in \cite{CKK}. In fact,
consistency with the Kawai-Yoshioka calculation
{\em forces} formula \eqref{DDD} for $\mathsf{R}_{j_L,j_R}^h$.

The refined invariants $\mathsf{R}_{j_L,j_R}^h$ may be viewed
as arising from the cohomology of 
$P_n(S,h)$, the moduli space  
of stable pairs on the $K3$ surface $S$ with positive irreducible
class $\alpha$ satisfying 
$$\langle \alpha, \alpha \rangle=2h-2\ .$$
Formula \eqref{DDD} also agrees with the 
refined invariant for $K3$ surfaces calculated in 
\cite{Huang:2013yta}.\footnote{%Different sign conventions were used in \cite{Huang:2013yta}. 
See  equations (8.3) and (8.4) of \cite{Huang:2013yta}  as 
well as Table 8 in Appendix D.2.}

Formula \eqref{DDD} uniquely determines $\mathsf{R}_{j_L,j_R}^h$
for $h\geq 0$. 
As  a consequence, the following
stabilization
property holds: {\em for fixed $i$ and $j$, the refined invariant
$\mathsf{R}^h_{\frac{h}{2}-i,\frac{h}{2}-j}$ is independent of $h$ for
sufficiently large $h$}.{\footnote{Refined invariants for local $\proj^2$  were found to stabilize in [CKK].  
We expect stabilization to hold more generally.}}

\begin{table}[h]
\begin{center} {% \footnotesize 
\begin{tabular} {|c|c|} \hline ${{\mathsf{R}^0_{\frac{i}{2},\frac{j}{2}}}_{}}$  & i=0 \\  \hline  j=0 & 1\\ \hline  \end{tabular}} 
\hspace{0.5cm}
\begin{tabular} {|c|cc|} \hline  ${\mathsf{R}^1_{\frac{i}{2},\frac{j}{2}}}_{}$  & j=0 & 1 \\  
\hline  i=0 & 20 &  \\ 1  &  & 1 \\  \hline \end{tabular} 
\hspace{0.5cm}
\begin{tabular} {|c|ccc|} \hline 
 ${\mathsf{R}^2_{\frac{i}{2},\frac{j}{2}}}_{}$   & j=0 & 1 & 2 \\  \hline i=0 & 231 &  &  \\  1 & 
 & 21 &  \\  2 &  &  & 1 \\  \hline \end{tabular} \hspace{0.0cm}
\end{center}
\end{table}
\vskip 0cm
\begin{table}[h]
\begin{center}
 \begin{tabular} {|c|cccc|} \hline  ${\mathsf{R}^3_{\frac{i}{2},\frac{j}{2}}}_{}$  & j=0 & 1 & 2 & 3 \\  \hline i=0 & 1981 &  & 1 &  \\  
1 &  & 252 &  &  \\  2 & 1 &  & 21 &  \\  3 &  &  & 
 & 1 \\  \hline \end{tabular} 
\hspace{0.5cm}
\begin{tabular} {|c|ccccc|} \hline  ${\mathsf{R}^4_{\frac{i}{2},\frac{j}{2}}}_{}$ 
 &j=0 & 
1 & 2 & 3 & 4 \\  \hline i=0 & 13938 &  & 21 &  &  \\  1 &  & 
2233 &  & 1 &  \\  2 & 21 &  & 253 &  &  \\  3 &  & 1 &  
& 21 &  \\  4 &  &  &  &  & 1 \\  \hline 
\end{tabular}
% \begin{table}[h]
% \begin{center} 
% \begin{tabular} {|c|cccccc|} \hline $R^4_{\frac{i}{2},\frac{j}{2}}$
% & j=0 & 1 & 2 & 3 & 4 & 5 \\  \hline i=0 & 84777 
%&  & 253 &  &  &  \\  1 &  & 16171 &  & 22 &  &  \\  2 & 
% 253 &  & 2254 &  & 1 &  \\ 3 &  & 22 &  & 253 &  &  \\  
% 4 &  &  & 1 &  & 21 &  \\  5 &  &  &  &  &  & 1 \\  \hline \end{tabular} 
\vskip .4cm
\caption{All nonvanishing $\mathsf{R}^{h}_{j_L,j_R}$ for $h\le 4$ for 
$K3$ surfaces}
\label{tableBettiK3}
\end{center}
\end{table}
\vskip -.7cm

\subsection{Unrefined BPS invariants}
Let $X$ be a nonsingular projective Calabi-Yau 3-fold, and
let $\beta\in H_2(X,\mathbb{Z})$ be a curve class.

The relation of the refined BPS invariants $\mathsf{N}_{j_Lj_R}^\beta$ 
of $X$ to the 
unrefined BPS invariants $\mathsf{n}_g^\beta$ of $X$ is obtained
by from their definitions in terms of 
traces on the BPS Hilbert space ${\cal H}_{BPS}$
arising from wrapping $M5$ branes on curves in $X$.
The Hilbert space ${\cal H}_{BPS}$ carries an 
$$SU(2)\times SU(2)= SU(2)_L \times SU(2)_R$$
action. 
We denote the
irreducible representations of $SU(2)\times SU(2)$ by pairs
$[j_L,j_R]$ where $j_L,j_R\in \frac{1}{2} \mathbb{Z}_{\geq 0}$. 
%In the physical definition,
%the representation $[j_L,j_R]$ labels the spin content of the 
%5d BPS states with multiplicity 
%$\mathsf{R}^h_{j_L,j_R}$.
The refined invariants arise via the formula:
\begin{equation}\label{zxx}
{\rm Tr}_{{\cal H}_{BPS}}\,y^{\sigma_3^L} u^{\sigma_3^R} q^H=
\\\sum_\beta \sum_{j_L,j_R\in\frac{1}{2} \mathbb{Z}_{\ge 0}} \mathsf{N}^\beta_{j_L,j_R}\, [j_L]_{y}\, [j_R]_{u}\, q^\beta \ .
\end{equation}         
Similarly, the unrefined invariants arise as: 
\begin{equation} \label{zxxx}
{\rm Tr}_{{\cal H}_{BPS}}\,(-1)^{F_R} y^{\sigma_3^L} q^H =\sum_\beta \sum_{g\in \mathbb{Z}_{\ge 0} } 
\mathsf{n}^\beta_{g} \left(y^\frac{1}{2}+y^{-\frac{1}{2}}\right)^{2 g} q^\beta \ .  
\end{equation}
Here, 
$(-1)^{F_R}$ acts as $(-1)^{2j_R}$ on $[j_R]$. 
The matrix
$$\sigma_3=\left(\begin{array}{cc}1&0\\ 0& -1\end{array}\right)$$
 is twice the Cartan 
element of $SU(2)$ measuring the spins of the BPS state.
The operator $H$ measures the mass of the BPS state.

Let $I_g$ denote the $SU(2)$ representation associated to the cohomology
of an abelian variety of dimension $g$,
$$I_g=\left(2[0]+\left[\frac{1}{2}\right]\right)^{\otimes \, g}\ .$$  
%A change of basis of $SU(2)$ representations is
%determined by: 
%$$
%\left[\frac{n}{2}\right]_{y^2}=\sum_{g=0}^n M_{n,g}   \left[I_g\right]_{y^2}\ .%$$
%Here, $M_{n,g}$ is
%a lower triangular matrices with ones along the diagonal:  
%$$M_{n,g}= (-1)^{n+g}\left(n+g-1\atop n-g\right),\quad (M^{-1})_{gn}=\left(2g\a%top g-n\right)-\left(2g\atop g-n-2\right)\ . $$
By taking the trace on $I_g$, we obtain the following identity:
$$ 
{\rm Tr}_{I_g} \, y^{\sigma_3}= \left(y^\frac{1}{2}+y^{-\frac{1}{2}}\right)^{2 g}\ .
%=\left[I_g\right]_{y^2}\ .
$$ 
Setting $u=-1$ in \eqref{zxx} and using \eqref{zxxx},
we obtain the basic relationship between the refined and
unrefined invariants:
\begin{equation} \label{tgg12}
\sum_{j_L,j_R\in \frac{1}{2} \mathbb{Z}_{\ge 0}} (-1)^{2j_R}(2 j_R+1) 
\mathsf{N}^\beta _{j_L, j_R}\, \left[j_L\right]=\sum_{g\in \mathbb{Z}_{\ge 0}} 
\mathsf{n}^\beta_g\, I_g\ . 
\end{equation}

If we specialize \eqref{tgg12} to the refined invariants
$\mathsf{R}^h_{j_L,j_R}$
of $K3$ surfaces and change variables
$$y\mapsto -y$$
in \eqref{DDD}, we obtain
\begin{equation*} 
\sum_{j_L,j_R\in \frac{1}{2} \mathbb{Z}_{\ge 0}} (2 j_R+1) 
\mathsf{R}^h _{j_L, j_R}\, \left[j_L\right]_{y}=\sum_{g\in \mathbb{Z}_{\ge 0}} 
(-1)^g\mathsf{r}^h_g\, \left(y^\frac{1}{2}-y^{-\frac{1}{2}}\right)^{2 g}
\ 
\end{equation*}
and recover the
KKV formula for the BPS counts $\mathsf{r}^h_g$
of $K3$ surfaces from \eqref{DDD},
\begin{multline*}
\sum_{h=0}^\infty \sum_{g\geq 0}
%(-1)^{2(j_L+j_R)}\mathsf{R}_{j_L,j_R}^h\ 
%[j_L]_{(-u)^2}[j_R]_{(-y)^2}
(-1)^g \mathsf{r}_{g}^h\ 
\left(y^\frac{1}{2}-y^{-\frac{1}{2}}\right)^{2 g}
\,
%\frac{(-u)^{-2j_L} - (-u)^{2j_L+2}}{1-u^2}
%\frac{(-y)^{-2j_R} - (-y)^{2j_R+2}}{1-q^2}
q^h= \\
\prod_{n=1}^\infty
 \frac{1}{
(1-y^{-1} q^n)(1-y q^n) (1-q^n)^{20}(1-y^{-1}q^n)
(1-y q^n)} \, ,
\end{multline*}
in accordance with \cite{gwnl,PT2}.

\vspace{+10pt}
\begin{table}[h]
\begin{center}
\begin{tabular}{|c|ccccc|}
        \hline
\textbf{}
$\mathsf{r}_{g}^h$&    $h= 0$ & 1  & 2 & 3 & 4 \\
        \hline 
$g=0$ & $1$ & $24$ & $324$ & 
$3200$ &$25650$  \\
1      &  & $-2$ & 
$-54$ & $-800$  & $-8550$      \\
2      & & & $3$ & 
$88$ & $1401$       \\
3      & &  & 
 & $-4$  & $-126$       \\
4      &  &  & 
 &   & 5       \\
       \hline
\end{tabular}
\vskip .4cm
\caption{All nonvanishing $\mathsf{r}^{h}_{g}$ for $h\le 4$ for  
$K3$ surfaces}
\end{center}
\end{table}

\subsection{Mathieu moonshine}
The important conceptual difference between the invariants
$\mathsf{N}^\beta_{j_L,j_R}$ and $\mathsf{n}^\beta_g$ 
is that the former are actual BPS degeneracies.  Hence,
we expect $\mathsf{N}_{j_L,j_R}^\beta$ to always be nonnegative.
Moreover, if there is a symmetry  acting 
on the Hilbert space which commutes with $H$  and the spin operators,
the multiplicities 
must  fall in representations of the symmetry group. 
For  $E_n$ del Pezzo surfaces, 
the invariants $\mathsf{N}^\beta_{j_rj_R}$ were observed in \cite{Huang:2013yta}
to fall naturally in representations of $E_n$.
The Weyl group of $E_n$ acts on the geometry 
by Lefshetz  monodromy. 

Greg Moore pointed out to us at {\em String Math 2014} during the 
presentation~\cite{Sheldonstalk} 
that the number $\mathsf{R}^2_{0,0}=231$ is  the dimension of an  
irreducible representation of the 
Matthieu group $\mathsf{M}_{24}$, a sporadic group of order $244823040$, 
which is  conjecturally \cite{Eguchi:2010ej} related to the elliptic genus of
the  
$K3$ surface. 
The dimensions of the 26 irreducible representations of 
the group $\mathsf{M}_{24}$ are
$$ 
\begin{array}{rl}
&1, 23, 45, 231, 252, 253, 483, 770, 990, 1035, 1265, 
\\& 
1771, 2024,  2277, 3312, 3520, 5313, 5544, 5796, 10395\ , 
\end{array}
$$
where the representations of dimension  $45,231,770,990,1035$ come 
in complex conjugated pairs. There is an extra real representation of 
dimension $1035$. 

We have checked all the values taken by refined invariants 
$\mathsf{R}^h_{j_L,j_R}$ in Table 1, except for the values
 $20$ (and $21=20+1$), are expressible 
in a simple  way in
terms of the dimensions of the irreducible
 representations of $\mathsf{M}_{24}$. Beside the dimensions  
which occur directly, we have:
$$\begin{array}{llrl}
{\mathsf{R}^3_{0,0}}_{}&=1981=2\cdot 990+1 \ \ \qquad  &   
\mathsf{R}^4_{0,0}&= 13938=2\cdot 5313+3312  \\ 
{\mathsf{R}^4_{\frac{1}{2},\frac{1}{2}}}_{}& =2233=2 \cdot 990+ 253  &  
\mathsf{R}^5_{1,1}&=2254=1771+483 \\ 
\mathsf{R}^6_{\frac{3}{2},\frac{3}{2}}&=2255=1265+990\, . &
\end{array}
$$  
Since the dimensions of the representations of $\mathsf{M}_{24}$ are 
small, the significance is somewhat limited.
The decompositions with a minimal 
numbers of summands given above are not always unique. For example, 
$$\mathsf{R}^6_{\frac{3}{2},\frac{3}{2}}=2024+231\, \text{ \ and \ } 
\mathsf{R}^4 _{0,0}=10395+3520+23\, .$$ 
Nevertheless, the action of the Matthieu group is expected{\footnote{See \cite{CCC,GG}
for related constructions and \cite{CDHK} for subsequent developments.}}
in the string 
compactification on $K3\times T^2$  and the refined spacetime BPS 
spectrum is a natural place to see its action.

\section{Conjecture C} \label{con2}

Our motivic convention is the following.{\footnote{Our conventions
here distribute the signs in a slightly different manner than the
conventions of \cite{CKK}, but there is no essential difference.}}
 If the moduli space
of stable pairs $P$ is a nonsingular variety of complex
dimension $d$, then
the associated motivic invariant is defined\footnote{The motivic invariant of $P$ may depend on a choice of orientation.  We have made an implicit choice here as explained in Section \eqref{vm} -- the only choice possible if 
$\text{Pic}(P)$ has no 2-torsion.}  to be
$$\mathsf{L}^{-\frac{d}{2}} [P]\ \in\ 
\mathsf{K}_{\text{var}}^{\widehat{\mu}}[\mathsf{L}^{-1}]\ .$$
Here, $[P]$ is the usual element associated to $P$ in
the Grothendieck ring of varieties.
When considering the virtual Poincar\'e polynomial,
$$\widetilde{\mathsf{H}}:\mathsf{K}_{\text{var}}^{\widehat{\mu}}(\mathsf{L}^{-1}) \rightarrow
 \mathbb{Z}[u,u^{-1}]\, ,$$
we map $\mathsf{L}^{\pm\frac{1}{2}} \mapsto u^{\pm 1}$.
We define
$\mathsf{H}(P)$ to be the virtual Poincar\'e polynomial of 
$\mathsf{L}^{-\frac{d}{2}} [P]$.

Let $\alpha\in \text{Pic}(S)$ be a positive irreducible class of norm
square
$$\langle \alpha,\alpha \rangle = 2h-2 \, .$$
Then $P_n(S,h)$ is nonsingular of dimension $2h-1+n$.
Hence, the relation
$${\mathsf{H}}(P_n(S,h))=u^{-2h+1-n}\, \widetilde{\mathsf{H}}(P_n(S,h))$$
of Section \ref{ky} is consistent with our motivic conventions.

An elementary verification based upon the interpretation of the
Kawai-Yoshioka calculation in Section \ref{ky} and the
definition of the refined invariants in Section \ref{gv} yields
the following identity: the $v^\alpha$ coefficient of the
product
\begin{equation} \label{fred}
\prod_{j_L,j_R,m_L,m_R,m,j}\,
%\prod_{m_L=-j_L}^{j_L}
%\prod_{m_R=-j_R}^{j_R}
%\prod_{m=1}^\infty
%\prod_{j=0}^{m-1}
\big(1+u^{-m+1+2j-2m_R}\, y^{m-2m_L}\, v^{\alpha}\big)^{(-1)^{2(j_L+j_R)}
\mathsf{R}^h_{j_L,j_R}}
\end{equation}
exactly equals
\begin{equation}\label{vv34}
\sum_{n=1-h}^\infty  
\mathsf{H}(P_n(S,h))\  y^n\ .
\end{equation}
The product $\prod_{j_L,j_R,m_L,m_R,m,j}$ in \eqref{fred} signifies
\begin{equation} \label{george}
\prod_{j_L \in \frac{1}{2}\mathbb{Z}_{\geq 0}}\
\prod_{j_L \in \frac{1}{2}\mathbb{Z}_{\geq 0}}\
\prod_{m_L=-j_L}^{j_L}\
\prod_{m_R=-j_R}^{j_R}\
\prod_{m=1}^\infty\
\prod_{j=0}^{m-1}
\end{equation}
where $m_L$ and $m_R$ increase by steps of 1.

The product \eqref{fred} occurs in the definition of the
refined invariants \cite[Equation (8.1)]{CKK}. 
The equality of \eqref{fred} and \eqref{vv34} is a 
geometric constraint verified by definition \eqref{DDD}.
In fact, definition \eqref{DDD} is uniquely determined
by the above constraint.

To state our last conjecture, let
$\alpha\in \text{Pic}(S)$ be a positive, primitive class of norm square
$$\langle \alpha,\alpha \rangle = 2h-2 \, .$$
We will consider the motivic partition function
for classes which are multiples of $\alpha$,
$$\mathsf{Z}_h =
\exp\left(\sum_{k=1}^\infty \sum_{n\in \mathbb{Z}} \mathsf{H}_{n,k\alpha} \ y^n v^{k\alpha}
\right)\ .$$
For fixed $k$, the motivic invariant $\mathsf{H}_{n,k\alpha}$ vanishes
for sufficiently negative $n$. 
Assuming Conjectures A and B, we rewrite the partition
function as
$$\mathsf{Z}_h =
\exp\left(\sum_{k=1}^\infty \sum_{n\in \mathbb{Z}} \mathsf{H}_{n,k,h[k]} \ y^n v^{k}
\right)\ $$
where we define
$$2h[k]-2 = \langle k \alpha, k\alpha\rangle = k^2(2h-2), \ \ \ \
h[k]= k^2(h-1)+1\ .$$
The variable $v^\alpha$ has now been replaced by just $v$.

\vspace{15pt}
\noindent {\bf Conjecture C.} {\em For all $h$,
the partition function 
$$\mathsf{Z}_h =
\exp\left(\sum_{k=1}^\infty \sum_{n\in\mathbb{Z}} \mathsf{H}_{n,k,h[k]} \ y^n v^{k}
\right)\ $$
equals the product}
$$\prod_{k=1}^\infty \, \prod_{j_L,j_R,m_L,m_R,m,j}\,
%\prod_{m_L=-j_L}^{j_L}
%\prod_{m_R=-j_R}^{j_R}
%\prod_{m=1}^\infty
%\prod_{j=0}^{m-1}
\big(1+u^{-m+1+2j-2m_R}\, y^{m-2m_L}\, v^{k}\big)^{(-1)^{2(j_L+j_R)}
\mathsf{R}^{h[k]}_{j_L,j_R}} \ .
$$

\vspace{15pt}

The product $\prod_{j_L,j_R,m_L,m_R,m,j}$ appearing in Conjecture C
is just as before \eqref{george}. Conjecture C determines every 
$\mathsf{H}_{n,k,h[k]}$
in terms of the refined Gopakumar-Vafa invariants
obtained from primitive class geometry.
Such a relation may be viewed as a divisibility invariance
property.

If we substitute $u=-1$ in $\mathsf{H}_{n,k\alpha}$,
we recover the stable pairs invariants
$R_{n,k\alpha}$ defined in \cite{PT2} for $K3$ surfaces.
An unwinding of the definitions then shows Conjecture
C implies the KKV conjecture (proven in \cite{PT2}) for stable
pairs invariants in all classes.

%%%%%%%%%%%%%%%%%%%%%%%%%%%%%%%%%%%%%%%%%%%%%%%%%%%%%%%%%%%%%

\section{First predictions}

\subsection{Virtual motives} \label{vm}
Before presenting examples, we 
 quickly review the theory of virtual motives following \cite{BJM,J}.

Joyce and collaborators introduce the notion of an {\em oriented d-critical locus\/} as a framework for
defining motivic invariants within classical (non-derived) algebraic geometry.
Moduli spaces of stable pairs carry such a structure.  
We review the aspects which are
most relevant for us and refer the reader to \cite{J} for 
the omitted details.

A {\em d-critical locus} is a variety $M$ which can locally be realized as 
$\mathrm{Crit}(f)$ for a holomorphic $f$ on a smooth space $U$ (a {\em critical
chart\/}),
with a weak notion of compatibility among the critical charts.  
The compatibility is strong enough however to
define  a {\em virtual canonical bundle\/} 
\begin{equation*}
K_M^{\mathrm{vir}}\in\mathrm{Pic}(M^\mathrm{red}).
\label{kvir}
\end{equation*}
\noindent Given a critical chart $(U,f)$, there is a canonical isomorphism
\begin{equation}
K_M^\mathrm{vir}|_{\mathrm{Crit}(f)}\simeq K_U^{\otimes2}|_{\mathrm{Crit}(f)}.
\label{canonisom}
\end{equation}  
An {\em orientation\/} is a choice of square root of the virtual canonical bundle
$$(K_M^{\mathrm{vir}})^{1/2}\in\mathrm{Pic}(M^{\mathrm{red}})\, 
 .$$  

From the data of a oriented $d$-critical locus,
a virtual motive $[M]^{\mathrm{vir}}$ can be defined.  There are two
ingredients:
\begin{itemize}
\item The motivic vanishing cycle of Denef and Loeser \cite{DL},
\item A principal $\bZ_2$ bundle determined by the choice of orientation.
\end{itemize}
A local virtual motive can be associated to the motivic vanishing cycle. After
a motivic twist by the principal $\bZ_2$ bundle, Joyce and collaborators
\cite{J} show  the local
motives glue together.

We review the motivic vanishing cycle  following \cite{DL}. 
For our examples, we will only require
the case where $f$ is of the form
\[
f=\prod_{i=1}^nz_i^{n_i}.
\]
\noindent Here $(z_1,\ldots,z_n)$ are coordinates in a neighborhood $U$ of the origin
in $\bC^n$.  Put $U_0=f^{-1}(0)$.

For an index set $I\subset\{1,\ldots,n\}$ define $E_I$ by the equations $z_i=0$
for all $i\in I$, and define 
$$E_I^\circ=E_I-\cup_{j\not\in I}E_j\, .$$
Let $m_I=\mathrm{gcd}\{n_i\mid i\in I\}$ and define $\widetilde{E}_I\to E_I^\circ$ by
\begin{equation*}
\widetilde{E}_I=\left\{(z,w)\in E_I^\circ\times\bC\mid w^{m_I}=\prod_{j\not\in I}
z_j^{n_j}
\right\}
\end{equation*}
\noindent with a natural projection to $E_I^\circ$.

The group $\mu_{m_I}$ of roots of unity acts on $\widetilde{E}_I$ by 
its action on $w$. In fact,
$\widetilde{E}_I$ is a Galois $\mu_{m_I}$-cover of $E_I^\circ$.
Denoting the action by $\rho_I$, we obtain
 an element $[\widetilde{E}_I,\rho_{I}]$ in
the ring of equivariant motives over $U$.  The motivic nearby
cycle of $f$ is
\begin{equation}\label{chad}
\mathrm{MF}_{U,f}^\mathrm{mot}=\sum_{I\neq \emptyset}\left(1-\bL
\right)^{|I|-1}[\widetilde{E}_I,\rho_{I}],
\end{equation}
where $\bL=[\bA^1]$ as before.

The {\em motivic vanishing cycle\/} of $f$ is
\begin{equation*}
\mathrm{MF}_{U,f}^{\mathrm{mot},\phi}=\bL^{-\dim U/2}\left([U_0]-
\mathrm{MF}_{U,f}^\mathrm{mot}
\right)
\end{equation*}
We only need the motivic vanishing cycle in three special cases:
\begin{itemize}
\item $f=0$
\item $f=z_1^2z_2^2$
\item $f=z_1^2z_2$
\end{itemize}

In case $f=0$, we have $X=U_0=U$ is nonsingular, $\mathrm{MF}_{U,f}^{\mathrm{mot}}$
is empty, and so the motivic vanishing cycle is $\bL^{-\dim U/2}[U]$.
Here, we match the conventions of Section \ref{con2}.

In case $f=z_1^2z_2^2$, for every nonempty $I\subset \{1,2\}$, we have
$m_I=2$ and $\widetilde{E}_I$ is a disconnected
double cover of $E_I^\circ$.  So each summand
in the motivic nearby cycle \eqref{chad} is the product of the
respective $E_I^\circ$ with the absolute equivariant motive
\begin{equation*}
[\mu_2,\rho]
\end{equation*}
\noindent where $\rho$ denotes the action of $\mu_2$ on itself.
%In the ring of motives, it is well known that 
%\[
%\left(1-[\mu_2,\rho]
%\right)^2=\bL,
%\]
%\noindent so we follow convention and write
%\begin{equation}
%\bL^{1/2}=1-[\mu_2,\rho].
%\label{l12}
%\end{equation}
We therefore obtain
$$
\mathrm{MF}_{U,f}^{\mathrm{mot}}=
 [\mu_2,\rho]\left(
[E_1^\circ]+[E_2^\circ]+\left(1-\bL\right)[E_{12}^\circ]
\right).
$$
Using \eqref{jq}, we find the following expression:
\begin{multline*}
\mathrm{MF}_{U,f}^{\mathrm{mot},\phi}= \\
\bL^{-\dim U/2}
\left(\bL^{1/2}\left([E_1^\circ]+[E_2^\circ]\right)+
\left(1-\left(1-\bL^{1/2}\right)\left(1-\bL\right)
\right)E_{12}^\circ
\right)
\end{multline*}
\noindent which simplifies to
\begin{equation}
\bL^{-\dim M /2}\left([M]+[E_{12}^\circ]\left(\bL^{1/2}-\bL\right)\right)\ .
\label{locx2y2vir}
\end{equation}

In case $f=z_1^2z_2$, we have $m_1=2$ and $m_I=1$ otherwise.  Thus, 
\[
\mathrm{MF}^{\mathrm{mot}}_{U,f}=[\widetilde{E}_1,\rho_1]+[E_2^{\circ}]+\left(1-\bL\right)
[E_{12}^\circ]
\]
and 
\begin{equation}
\mathrm{MF}^{\mathrm{mot},\phi}_{U,f}=\bL^{-\dim U/2}\left([E_1^\circ]-[\widetilde{E}_1,\rho_1]
+\bL [E_{12}^\circ].
\right)
\label{locx2yvir}
\end{equation}
As expected, $E_2^\circ$ has cancelled out ($E_2^\circ$ is not
part of $\mathrm{Crit}(f)$).

The principal $\bZ_2$ bundle associated with a choice of orientation is given by
the local isomorphisms
\[
\left(K_M^{\mathrm{vir}}\right)^{1/2}|_{\mathrm{Crit}(f)}\simeq K_U|_{\mathrm{Crit}(f)}
\]
\noindent which are square roots of the canonical isomorphism (\ref{canonisom}).

In the case $f=0$ or $f=z_1^2z_2^2$, the principal 
bundle $\bZ_2$ bundle is trivial in a punctured
neighborhood of each $E_I^\circ$, essentially since there is no
ramification in the Galois $\mu_2$ covers described above.  As we shall see,
the principal $\bZ_2$ bundle plays an important role in the case 
$f=z_1^2z_2$.

Next we globalize the $f=0$ geometry. Suppose $M$ is nonsingular.  
As a $d$-critical locus, $M$ can be described
by a single critical chart $(M,0)$, so 
$$K_M^{\mathrm{vir}}\cong K_M^{\otimes2}\ .$$  For 
the natural choice of orientation $K_M$, we have
\begin{equation*}
[M]^\mathrm{vir}=\bL^{-\dim M /2}[M],
\end{equation*}
 If there is no 2-torsion in
$\mathrm{Pic}(M)$, then $K_M$ is the only possible orientation, as
will be the case in our example.

In a second geometry which will arise, 
$M^{\mathrm{red}}$ is a union of two
nonsingular components $E_1,E_2$
meeting transversally along a  nonsingular irreducible divisor 
$E_{12}$.  In such case, $M$ {\em must} be nonreduced.  
We will also have $E_1-E_2$
and $E_2-E_1$ nonsingular, so nilpotents can occur only along $E_{12}$.
We {\em assume\/} the simplest possible 
scheme structure compatible with the situation: $(z_1^2z_2,z_1z_2^2)$  in the neighborhood of any point of any $E_{ij}$, where $z_1=0$
and $z_2=0$ are local equations for $E_1$ and $E_2$ respectively.  
In other words, we take 
\[
f=z_1^2z_2^2
\]
as the superpotential.

Consider the natural isomorphism
\[
K_M^{\mathrm{vir}}|_{E_1^\circ}\cong K_{E_1^\circ}^{\otimes2}
\]
These bundles extend to respective line bundles
$K_M^{\mathrm{vir}}|_{E_1}$ and $K_{E_1}^{\otimes2}$ on $E_1$.
Direct computation shows the isomorphism vanishes to order~2 along $E_{12}$.
So
\begin{equation*}
K_M^{\mathrm{vir}}|_{E_1}\cong K_{E_1}^{\otimes2}\left(-2E_{12},
\right)
\end{equation*}
\noindent with the analogous identification on $E_2$.

So there is again a natural orientation $(K_M^{\mathrm{vir}})^{1/2}$ determined by
\begin{equation*}
\left(K_M^{\mathrm{vir}}\right)^{1/2}|_{E_i}\cong K_{E_i}\left(-E_{12}
\right)
\end{equation*}
\noindent for $i=1,2$.  If in addition there is no 2-torsion in $\mathrm{Pic}(E_1)$
or $\mathrm{Pic}(E_2)$, then the orientation is unique.  
Such uniqueness will occur in our example.

With the unique orientation, we have a globalization of (\ref{locx2y2vir}),
\begin{equation}
[M]^{\mathrm{vir}}=\bL^{-\dim M/2}\big([M]+[E_{12}]\left(\bL^{1/2}-\bL\right)\big).
\label{x2y2vir}
\end{equation}

A third geometry will arise:  $M^\mathrm{red}$ 
is irreducible,
nonsingular, and contains a nonsingular divisor $D\subset M$ 
{\em precisely\/} along which  $M$ is nonreduced.  We {\em assume\/} the simplest possible
scheme structure compatible with the situation, $(z_1^2,z_1z_2)$ in the 
neighborhood of any point of $D$, where $z_1=0$ is a local equation for
$M^{\mathrm{red}}$ and $z_1=z_2=0$ are local equations for $D$.  In other
words, we take
\[
f=z_1^2z_2
\]
as the superpotential,{\footnote{The superpotential
$z_1^2z_2$ together with the associated $\mathbb{Z}_2$-monodromy
was first analyzed in \cite[Example 4.5]{Sz} to calculate nontrivial
refined stable
pairs invariants of local $\proj^1$.}} locally identifying $E_1$ with $M^{\mathrm{red}}$ and
$E_{12}$ with $D$.

Consider the natural isomorphism
\[
K_M^{\mathrm{vir}}|_{M^{\mathrm{red}}-D}\cong K_{M^{\mathrm{red}}-D}^{\otimes2}
\]
These bundles extend to respective line bundles
$K_M^{\mathrm{vir}}$ and $K_{M^{\mathrm{red}}}^{\otimes2}$ on $M^{\mathrm{red}}$.
Direct computation shows the isomorphism vanishes to order~1 along $D$.
So
\begin{equation}
K_M^{\mathrm{vir}}\cong K_{M^{\mathrm{red}}}^{\otimes2}\left(-D
\right).
\label{kvirx2y}
\end{equation}
It is apparent that there is no natural orientation 
$(K_M^{\mathrm{vir}})^{1/2}$ in this general
situation.

For the moduli space of stable pairs, 
we know that $M$ is an oriented $d$-critical locus by general theory.
Hence, we
conclude $D$ must be  even:
$$\OO_{M^{\mathrm{red}}}(D)\simeq L^{\otimes2}$$ for some line bundle $L$ on 
$M^\mathrm{red}$.\footnote{The same calculation was applied in 
\cite[Example~2.39]{J} to the situation $$(M^{\mathrm{red}},D)
=(\bP^1,p)$$  to show that a certain $d$-critical locus was {\em not\/}
orientable since the class of a point is not even in $\mathrm{Pic}(\bP^1)$.}  Then, we have an orientation
\[
\left(K_M^{\mathrm{vir}}\right)^{1/2}=K_{M^{\mathrm{red}}}\otimes L^{-1}\, ,
\]
which is the only possibility if $\text{Pic}(M)$
has no 2-torsion.

In the above oriented situation, the principal $\bZ_2$ bundle of square roots of
(\ref{canonisom}) in a critical chart naturally ramifies when extended to $D$.
Let $$\pi:\widetilde{M}\to M$$ be the double cover of $M$
branched along $D$, with the natural involution $\iota$ and ramification
divisor $\tilde{D}$.  The principal
$\bZ_2$ bundle modifies the local virtual motive 
$\bL^{-\dim M/2}[M-D]$ of $M-D$ to\footnote{In \cite{BJM}, the $\bZ_2$ twists
are only defined in a quotient of the equivariant motivic ring.  We presume 
the computation holds in the equivariant motivic ring itself if
other approaches to virtual motives are followed \cite{KS}.}  
%In any
%case, the virtual Poincar\'e polynomial is not affected by this issue.}
\[
\bL^{-(\dim M+1)/2}\left([M-D]-[\widetilde{M}-\widetilde{D},\iota]
\right).
\]
Comparison with (\ref{locx2yvir}) shows how to extend the motive 
globally.  The result is
\begin{equation}
\bL^{-(\dim M+1)/2}\left([M-D]-[\widetilde{M}-\widetilde{D},\iota]+\bL[D]
\right)
\label{mlocx2yvir}
\end{equation}

We now compute the virtual Poincar\'e polynomial of (\ref{mlocx2yvir}).
To convert a $\bZ_2$-equivariant motive $[V,\iota]$ to a 
virtual Poincar\'e polynomial, 
decompose $H^*_{\mathrm{c}}(V)$ into its even and odd parts under $\iota^*$:
\[
H^*_{\mathrm{c}}(V)=H^+_{\mathrm{c}}(V)\oplus H^-_{\mathrm{c}}(V)
\]
\noindent and then take the virtual Poincar\'e polynomial, which we write 
as
\[
\widetilde{\mathsf{H}}\left(V\right)=\widetilde{\mathsf{H}}^+_{\mathrm{c}}(V)+
\widetilde{\mathsf{H}}^-_{\mathrm{c}}(V)
\]

In our conventions the virtual Poincar\'e
polynomial of $[V,\iota]$ is then 
\begin{equation}
\widetilde{\mathsf{H}}(V,\iota)=
\widetilde{\mathsf{H}}^+_{\mathrm{c}}(V)-u
\widetilde{\mathsf{H}}^-_{\mathrm{c}}(V)\ ,
\label{mu2pp}
\end{equation}
see \cite[Corollary 7.2]{L}. 
The minus sign in (\ref{mu2pp}) 
is consistent with the evaluation
$$\widetilde{\mathsf{H}}(\bL^{1/2})= \widetilde{\mathsf{H}}
\big(1-[\mu_2,\rho]\big)=u\,
. $$
%with $\bL^{1/2}$ defined by (\ref{l12}).
%It  is easier to state the result in terms of the virtual
%Poincar\'e polynomial.  

Finally, the virtual Poincar\'e polynomial of (\ref{mlocx2yvir}) is
\begin{equation}
u^{-(\dim M+1)}\left(u\widetilde{\mathsf{H}}_{\mathrm{c}}^-(\widetilde{M})+u^2\widetilde{\mathsf{H}}_c(D)
\right).
\label{x2ypp}
\end{equation}
%\noindent Putting $H^-(\tilde{M})=u^{-\dim M}\widetilde{H^-}(\tilde{M})$, this
%expression simplifies to\footnote{The computation is facilitated by \cite[Cor.~7.2]{L}.}

%\noindent and the local virtual motives now glue to
%\begin{equation}
%\bL^{-\dim M/2}\left([M^{\mathrm{red}}]+\left(\bL-1\right)[D]\right).
%\label{virmotx2y}
%\end{equation}

\subsection{Poincar\'e polynomials}
Some elucidation of (\ref{mu2pp}) is in order here.  Let $K^{\widehat{\mu}}_0(\mathrm{HS})$ denote the Grothendieck ring of 
the category of Hodge structures with a $\widehat{\mu}$-action.  
There is a Hodge
characteristic map \cite{L},
\begin{equation*}
  \chi_{\mathrm{h}}:K^{\widehat{\mu}}_{\mathrm{var}}\to K^{\widehat{\mu}}_0(\mathrm{HS})\, ,
\end{equation*}
which can be extended to
\begin{equation}
  \chi_\mathrm{h}:K^{\widehat{\mu}}_{\mathrm{var}}[\bL^{-1}]\to K^{\widehat{\mu}}_0(\mathrm{HS})
\end{equation}
since $\chi_{\mathrm{h}}(\bL)$ is invertible.  We will define a {\em virtual
Poincar\'e
polynomial} map
\begin{equation*}
 \mathsf{P}:K^{\widehat{\mu}}_0(\mathrm{HS})\to \ZZ[u,u^{-1}]
\end{equation*}
which then determines a virtual Poincar\'e 
polynomial{\footnote{Equivalent definitions have appeared before, see 
\cite[Appendix A.4]{SB} and the references there.}} 
map 
\begin{equation*}
\widetilde{\mathsf{H}}=\mathsf{P}\circ \chi_h:K^{\widehat{\mu}}_{\mathrm{var}}[\LL^{-1}]\to \ZZ[u,u^{-1}]\, .  
\end{equation*}
The definition of $\mathsf{P}$ is chosen so that
$\widetilde{\mathsf{H}}$ is a ring homomorphism.

For simplicity of exposition, we focus on the special case of 
$\mu_2$-equivariant Hodge structures (the only case which appears in 
the examples considered here) and say a few words about how to extend to
the general case.

By \cite[Cor.~7.2]{L}, the motivic convolution product in 
$K^{{\widehat{\mu}}}_{\mathrm{var}}$ descends
to a product in $K^{{\widehat{\mu}}}_0(\mathrm{HS})$, denoted $*$.  Restricting
to elements of $K^{{\mu_2}}_0(\mathrm{HS})$, the product is shown to  satisfy
\begin{multline}
  H*H'=
\\ H^+\otimes (H')^+ + H^+\otimes (H')^-+H^-\otimes (H')^++H^-\otimes (H')^-(-1),
\label{mu2conv}
\end{multline}
where $H^\pm$ are the even and odd parts of the $\mu_2$-action on $H$ (and 
similarly for $H'$). 

Forgetting the $\mu_2$ action, we let $\mathsf{Q}(H)$ be the Poincar\'e polynomial
of a Hodge structure $H$. 
Equation (\ref{mu2pp}) can be rephrased as defining $\mathsf{P}$ to be 
\begin{equation*}
  \mathsf{P}(H)=\mathsf{Q}(H^+)-u\mathsf{Q}(H^-).
\end{equation*}
Then (\ref{mu2conv}) implies $\mathsf{P}(H*H')$ is 
\begin{equation*}
\mathsf{Q}(H^+\otimes (H')^+) -u \mathsf{Q}(H^+\otimes (H')^-)-u(H^-\otimes (H')^+)
+u^2\mathsf{Q}(H^-\otimes (H'))
\end{equation*}
which equals $\mathsf{P}(H)\mathsf{P}(H')$ as desired.

The full
result \cite[Cor.~7.2]{L} extends (\ref{mu2conv}) to a formula for 
$*$ valid for any $$H,H'\in K^{{\widehat{\mu}}}_0(\mathrm{HS})\, ,$$
 expressed in terms
of characters of $\widehat{\mu}$.  Then $\mathsf{P}(H)$ 
can be defined by $\mathsf{Q}(H)$
for $H$ a trivial representation of $\widehat{\mu}$ and $-u\mathsf{Q}(H)$ if $H$
transforms by a nontrivial character of $\widehat{\mu}$.  The verification
of $$\mathsf{P}(H*H')=\mathsf{P}(H)\mathsf{P}(H')$$ in the general case is a bit
more involved, relying on the precise form
of \cite[Cor.~7.2]{L} and of the computation of the characters of the
cohomology of the Fermat curves used in the definition of motivic
convolution \cite{SK}.  The result from \cite{SK} also appears as
\cite[Lemma~7.1]{L}.

\subsection{Elliptically fibered K3 surfaces}
Let $S$ be an elliptically fibered $K3$ surface, 
$$\pi: S \rightarrow \proj^1\ ,$$
with section $s$ and fiber class
$f$.  No singular point of any fiber lies on 
$s$. 
  We will compute motivic stable pair invariants in classes
$$s,f,s+f\in \text{Pic}(S)$$ with small Euler characteristic following the method
of \cite{J}.
For our definition in Section \ref{spm}, we will consider various
families 
$$\epsilon: T \rightarrow (\Delta,0)$$
depending upon the class.

We start with the fiber class $f\in \text{Pic}(S)$.  If $(F,\tau)$ is a stable pair of class
$[F]=f$ and
$\chi(F)=0$, then $F={\mathcal{O}}_E$ for some fiber $E$ of 
$\pi$. Since the fibers are parametrized by $\bP^1$,
\[
P_0(S,f)\simeq\bP^1
\]
\noindent and
\[
[P_0(S,f)]^{{\rm vir}}=\bL^{-1/2}[\bP^1].
\]
Note  $\mathrm{Pic}(\bP^1)$ has no torsion, so there is no choice in the
motivic invariant.

If $\chi(F)=1$, then the $\mathrm{coker}(\tau)$ of the stable pair $(F,\tau)$ 
is a point, and 
$P_1(S,f)\cong S$, so
\[
[P_1(S,f)]^{{\rm vir}}=\bL^{-1}[S].
\]
Again, $\mathrm{Pic}(S)$ has no torsion, so there is no choice in the
motivic invariant.

%Similarly, $P_2(S,f)\simeq S^{[2]}$ the relative Hilbert scheme of
%length~2 for $S/\bP^1$.  I haven't checked if $S^{[2]}$ is smooth or not.
%For simplicity, let's proceed as if it is (I can get around this assumption) so that 
%\[
%[P_2(S,f)]^{{\rm vir}}=\bL^{-3/2}[S^{[2]}]
%\]
%\[
%Z^{\rm{mot}}=1+Q^f\left(\bL^{-1/2}[\bP^1]+\bL^{-1}[S]q+\bL^{-3/2}[S^{[2]}]q^2+\ldots
%\right)+\ldots
%\]

Next consider the section $s\in \text{Pic}(S)$.  
We see $P_1(S,s)$ is a point and 
$$P_2(S,s)\cong\bP^1\ .$$ Hence,
the coefficient of $v^s$ in the motivic partition function 
$\mathsf{Z}^{\text mot}$ is
\[
q+\bL^{-1/2}[\bP^1]q^2+\ldots \ .
\]
Since $s,f\in \text{Pic}(S)$ are irreducible classes, for any 1-rigid family $\epsilon$,
the above moduli identifications are valid on $T$.

The class $s+f$ is primitive but {\em not}
irreducible. In the Euler characteristic 0 case,  
$$P_0(S,s+f)\cong \bP^1\, $$ is
determined by the location of a fiber of $\pi$. Hence,
\[
[P_0(S,s+f)]^{\mathrm{vir}}=\bL^{-1/2}[\bP^1]\ .
\]
Again, $T$ plays no interesting role.

The more interesting geometry occurs in Euler characteristic 1.
The moduli space
$P_1(T,s+f)$ has two components, $E_1$ where the point is on the fiber and $E_2$
where the point is on $s$:
\begin{itemize}
\item $E_1\cong S$
\item $E_2\cong \bP^1\times\bP^1$
\end{itemize}
\noindent For $E_2$, the first
$\bP^1$ parametrizes the point of the fiber and the second
$\bP^1$ parametrizes the location of the point $\mathrm{coker}(\tau)$ 
on the section.
The two components meet along 
$$E_1\cap E_2\cong \bP^1$$ 
embedded in $S$ as the section $s$
and in $\bP^1\times\bP^1$ as the diagonal.  

We assume the local superpotential is the second 
form discussed in Section \ref{vm}.
By (\ref{x2y2vir}), we have
\[
[P_1(X,s+f)]^{\mathrm{vir}}=\bL^{-1}\left([P_1(X,s+f)]+[\bP^1](\bL^{1/2}-\bL )
\right)
\]

For the classes $s,f,s+f\in \text{Pic}(S)$ to order $q$, the motivic partition function is 
\begin{eqnarray*}
\mathsf{Z}^{\text{mot}}&=&1
+v^s\left(q+\ldots
\right)
\\ & & \ \ +\ v^f\left(\bL^{-1/2}[\bP^1]+q\bL^{-1}[S]%+\bL^{-3/2}[S^{[2]}]q^2
+\ldots
\right)  \\ & & \ \ +\ 
v^{s+f}\Big(\bL^{-1/2}[\bP^1]+
q\bL^{-1}\big([P_1(X,s+f)]+[\bP^1](\bL^{1/2}-\bL)\big) + \ldots
\Big)
\\ & & \ \ + \ldots\ .
\end{eqnarray*}
We now calculate the coefficient of $v^{s+f}$ in
$\log(Z^{\text{mot}})$. The $q^0$ coefficient is simply $\bL^{-1/2}[\bP^1]$.
The $q$ coefficient is
\begin{multline*}
\bL^{-1}\Big([P_1(X,s+f)]+[\bP^1]\left(\bL^{1/2}-\bL
\right)\Big)-\bL^{-1/2}[\bP^1]\\ =
\bL^{-1}\Big([S]+[\bP^1\times\bP^1]-[\bP^1]+[\bP^1]\left(\bL^{1/2}-\bL
\right)\Big)-\bL^{-1/2}[\bP^1]\\ =
\bL^{-1}[S]+\bL^{-1}[\bP^1]\big([\bP^1]-1+\bL^{1/2}-\bL\big)-\bL^{-1/2}
[\bP^1]\ .
\end{multline*}
The last expression is easily simplified to 
$\bL^{-1}[S]$.
We conclude the $q^0$ and $q^1$ coefficients of $v^{s+f}$
in $\log(Z^{\text{mot}})$ agree with the 
 $q^0$ and $q^1$ coefficients of $v^f$ in 
$\log(Z^{\text{mot}})$.

The above calculation provides nontrivial evidence for Conjectures
A and B. In fact, if more naive approaches to the motivic theory
are taken (for example using the actual moduli spaces or
even the Behrend function on the moduli spaces), the agreement
we have found fails.
We have verified the prediction to order $q^2$, but we do not
include the more involved calculations here. 

We next turn to the class $2f$.  We have 
$$P_0(S,2f)\cong\mathrm{Sym}^2(\bP^1)
\cong\bP^2\ .$$  However, in the algebraic twistor family, 
the scheme structure is {\em not}  reduced{\footnote{We thank R. Thomas for
the verification.}}  precisely along the diagonal
curve $D\subset \bP^2$, a plane conic.  The discussion of Section~\ref{vm} therefore
applies.  The virtual canonical bundle is $$K_{\bP^2}^{\otimes2}(-D)\cong\OO_{\bP^2}(-8)\, ,$$ and
 the algebraic twistor family is uniquely oriented by $\OO_{\bP^2}(-4)$.

The double cover of $\bP^2$ branched along $D$ is a nonsingular quadric surface $Q$
containing an isomorphic copy $\widetilde{D}$ of $D$.  The motive of 
$D\cong\widetilde{D}$ is $\bL+1$, so
the motive of  $Q-\widetilde{D}$ is $\bL^2+\bL$ and the motive of 
$\bP^2-D$ is $\bL^2$.
Therefore the odd part of the motive of $Q-\widetilde{D}$ is $\bL$.  Hence \eqref{x2ypp} yields
\[
\mathsf{H}_{0,2,0[2]}=\widetilde{\mathsf{H}}\left([P_0(X,2f)]^{\mathrm{vir}}\right)=1+\left(u+u^{-1}\right),
\]
\noindent in complete agreement with Conjecture~C. 

%For $P_2(S,f+s)$, there are three components
%\begin{itemize}
%\item $E_1\simeq S^{[2]}$ (two points on a fiber)
%\item $E_2\simeq (S\times\bP^1)^\sim$ (point on a fiber and point on section
%\item $E_3\simeq (\bP^1\times\bP^2)^\sim$ (two points on section)
%\end{itemize}
%where $(S\times\bP^1)^\sim$ is the blowup of $S\times\bP^1$ along  $\bP^1$, embedded in
%$S$ as the section (there is a $\bP^1$ worth of length 2 subschemes of $f+s$ supported at
%the intersection point).  $\bP^2$ is $(\bP^1)^{[2]}$ and $(\bP^1\times\bP^2)^\sim$ is the blowup
%of $\bP^1\times\bP^2$ along $\bP^1$, where $\bP^1$ is embedded in $\bP^2$ as the locus of
%length~2 subschemes.
%
%The intersections are $E_{12}\simeq S$, corresponding to one point on the fiber and the
%point $f\cap s$, and $E_{23}\simeq \bP^1\times\bP^1$.
%
%\[
%[P_2(S,f+s)]^{\mathrm{vir}}=\bL^{-3/2}\left([P_2(S,f+s)]+([S]+[\bP^1\times\bP^1](\bL^{1/2}-\bL)
%\right)
%\]
% 

  %%%%%%%%%%%%%%%%%%%%%%%%%%%%%%%%%%%%%%%%%%%%%%%%%%%%%%%%%%%%%
\section{Duality and 
Noether-Lefschetz theory} 
\label{B-model}
\subsection{Heterotic-Type II duality}
The Yau-Zaslow conjecture originates in heterotic-Type 
II duality in 6d, where the heterotic string is 
compactified on the four torus $T^4$ and the 
Type II string on $K3$. By the adiabatic argument~\cite{Vafa:1995gm}, 
this can be extended to 4d $N=2$ supersymmetric theories, which are obtained 
from {\sl dual pairs} of heterotic string compactifications on 
$K3\times T^2$ and Type II string compactifications on Calabi-Yau 3-folds $X$. 
The latter are $K3$ fibations over $\mathbb{P}^1$.  This construction of 
dual pairs requires a match between the vector- and hypermultiplet moduli 
spaces of the heterotic and the type II compactifications. The heterotic 
moduli parametrize the metric of $K3\times T^2$, the bundle data of the heterotic
compactification, and the heterotic dilaton $S$, which is in a vector multiplet.  
In Type IIA compactifications, the complexified K\"ahler moduli $(t,S)$ of $X$ parametrize 
the vector multiplet moduli space\footnote{The 
complex moduli of $X$  together with the Ramond fields  and the type II 
dilaton parametrize the hypermultiplet moduli space, which is of 
quaternionic dimension $h^{21}(X)+1$. This makes the duality much richer, 
but we focus on the vector moduli.}, which is  of complex dimension $h^{11}(X)$. 
In particular, the heterotic dilaton 
$$S=\frac{4 \pi }{g^2_{het}}+ i \theta$$ is identified 
with the complexified volume of the base $\mathbb{P}^1$. 

The simplest example is the 
STU-model, see~\cite{KMPS} for review. Here, the Calabi-Yau 3-fold 
$X$
is an elliptic fibration over 
$\mathbb{P}^1\times \mathbb{P}^1$  and a $K3$ fibration over $\mathbb{P}^1$  
with 
$$h^{11}(X)=3\ .$$
The three vector moduli 
are identified on the heterotic side with the heterotic dilaton $S$, the 
complex modulus $T$,  and the complexified K\"ahler modulus $U$ 
of the heterotic torus $T^2$ \footnote{As $h^{21}(X)=243$,  the  
heterotic hypermultiplet moduli space is of quaternionic dimension $244$.}. 

An impressive consequence of the proposed duality is that a perturbative 
heterotic one-loop amplitude predicts  {\sl all higher} genus 
amplitudes 
$$F(\lambda,t)=\sum_{g} \lambda^{2g-2} F_g(t)$$  of $X$ in 
the infinite base limit (the dependence on the $K3$ fiber classes). 
For the STU model, 
%---------------
\begin{equation}
\label{het/II}
 \lim_{S\rightarrow \infty} F(\lambda,S,T,U )=F^{1-loop}_{het}(\lambda,T,U)\ . 
\end{equation}
%---------------
We will use this relation in the holomorphic limit to extend Conjecture C to 
a proposal for the refined invariants of the STU model.
% (and more generally of
%$K3$-fibered Calabi-Yau 3-folds).

\subsection{Refined Noether-Lefschetz theory}
\subsubsection{Overview}
We pass now  from the  string point of view to the more precise 
mathematical perspective advanced in \cite{gwnl,PT2} as Noether-Lefschetz 
correspondences. For our study of refined invariants, the Noether-Lefschetz
numbers of \cite{gwnl,PT2} also require refinement.

\subsubsection{$\Lambda$-polarization}
Following the notation of \cite[Section 1.1]{PT2},
let $\Lambda$ be a fixed rank $r$  
primitive{\footnote{A sublattice
is primitive if the quotient is torsion free.}}
sublattice
\begin{equation*} 
\Lambda \subset U\oplus U \oplus U \oplus E_8(-1) \oplus E_8(-1)
\end{equation*}
with signature $(1,r-1)$, and
let 
$v_1,\ldots, v_r \in \Lambda$ be an integral basis.
The discriminant is
$$\Delta(\Lambda) = (-1)^{r-1} \det
\begin{pmatrix}
\langle v_{1},v_{1}\rangle & \cdots & \langle v_{1},v_{r}\rangle  \\
\vdots & \ddots & \vdots \\
\langle v_{r},v_{1}\rangle & \cdots & \langle v_{r},v_{r}\rangle 
\end{pmatrix}\ .$$
The sign is chosen so $\Delta(\Lambda)>0$.

A {\em $\Lambda$-polarization} of a $K3$ surface $S$   
is a primitive embedding
$$j: \Lambda \rightarrow \mathrm{Pic}(S)$$ 
satisfying two properties:
\begin{enumerate}
\item[(i)] the lattice pairs 
$\Lambda \subset U^3\oplus E_8(-1)^2$ and
$\Lambda\subset 
H^2(S,\mathbb{Z})$ are isomorphic
 via an isometry which restricts to the identity on $\Lambda$,
 \item[(ii)]
$\text{Im}(j)$ contains
a {quasi-polarization}. 
\end{enumerate}
By (ii), every $\Lambda$-polarized $K3$ surface is algebraic.

The period domain $M$ of Hodge structures of type $(1,20,1)$ on the lattice 
$U^3 \oplus E_8(-1)^2$   is
an analytic open set
 of the 20-dimensional  nonsingular isotropic 
quadric $Q$,
$$M\subset Q\subset \proj\big(    (U^3 \oplus E_8(-1)^2 )    
\otimes_\Z \com\big).$$
Let $M_\Lambda\subset M$ be the locus of vectors orthogonal to 
the entire sublattice $\Lambda \subset U^3 \oplus E_8(-1)^2$.

Let $\Gamma$ be the isometry group of the lattice 
$U^3 \oplus E_8(-1)^2$, and let
 $$\Gamma_\Lambda \subset \Gamma$$ be the
subgroup  restricting to the identity on $\Lambda$.
By global Torelli,
the moduli space $\mathcal{M}_{\Lambda}$ 
of $\Lambda$-polarized $K3$ surfaces 
is the quotient
$$\mathcal{M}_\Lambda = M_\Lambda/\Gamma_\Lambda.$$
We refer the reader to \cite{dolga} for a detailed
discussion.

\subsubsection{Noether-Lefschetz divisors}

Let $(\mathbb{L}, \iota)$ be a rank $r+1$
lattice  $\mathbb{L}$
with an even symmetric bilinear form $\langle\, ,\rangle$ and a primitive embedding
$$\iota: \Lambda \rightarrow \mathbb{L}.$$
Two data sets 
$(\mathbb{L},\iota)$ and $(\mathbb{L}',  \iota')$
are isomorphic if and only if there exist an isometry relating $\mathbb{L}$
and $\mathbb{L}'$ 
which takes $\iota$ to $\iota'$.
The first invariant of the data $(\mathbb{L}, \iota)$ is
the discriminant
 $\Delta \in \mathbb{Z}$ of 
$\mathbb{L}$.

An additional invariant of $(\mathbb{L}, \iota)$ can be 
obtained by considering 
any vector $v\in \mathbb{L}$ for which{\footnote{Here, $\oplus$
is used just for the additive structure (not orthogonal
direct sum).}}
\begin{equation}\label{ccff} 
\mathbb{L} = \iota(\Lambda) \oplus \mathbb{Z}v.
\end{equation}
The pairing
$$\langle v,\cdot\rangle : \Lambda \rightarrow \mathbb{Z}$$
determines an element of $\delta_v\in \Lambda^*$.
Let 
$G = \Lambda^{*}/\Lambda$
be the quotient defined via the injection
$\Lambda \rightarrow \Lambda^*$
 obtained from the pairing $\langle\, ,\rangle$ on $\Lambda$.
The group $G$ is abelian of order given by  the 
discriminant $|\Delta(\Lambda)|$.
The image 
$$\delta \in G/\pm$$
of $\delta_v$ is easily seen to be independent of $v$ satisfying 
\eqref{ccff}. The invariant $\delta$ is the {\em coset} of $(\mathbb{L},\iota)$

By elementary arguments, two data sets $(\mathbb{L},\iota)$ and $(\mathbb{L}',\iota')$
of  rank $r+1$ are isomorphic if and only if the discriminants and cosets are
equal.

Let $v_1,\ldots, v_r$ be an integral basis of $\Lambda$ as before.
The pairing of $\mathbb{L}$ 
with respect to an extended basis $v_{1}, \dots, v_{r},v$
is encoded in the matrix
$$\mathbb{L}_{h,d_{1},\dots,d_{r}} = 
\begin{pmatrix}
\langle v_{1},v_{1}\rangle & \cdots & \langle v_{1},v_{r}\rangle & d_{1} \\
\vdots & \ddots & \vdots & \vdots\\
\langle v_{r},v_{1}\rangle & \cdots & \langle v_{r},v_{r}\rangle & d_{r}\\
d_{1} & \cdots & d_{r} & 2h-2
\end{pmatrix}.$$
The discriminant is
$$\Delta(h,d_{1},\dots,d_{r}) 
= (-1)^r\mathrm{det}(\mathbb{L}_{h,d_{1},\dots,d_{r}}).$$
The coset $\delta(h, d_{1},\dots,d_{r})$ is represented by the functional
$$v_i \mapsto d_i.$$

The Noether-Lefschetz divisor $P_{\Delta,\delta} \subset \mathcal{M}_{\Lambda}$
is the closure of the locus
 of $\Lambda$-polarized $K3$ surfaces $S$ for which
$(\mathrm{Pic}(S),j)$ has rank $r+1$, discriminant $\Delta$, and coset $\delta$.
By the Hodge index theorem{\footnote{The intersection
form on $\mathrm{Pic}(S)$ is
nondegenerate for an algebraic $K3$ surface.
Hence, a rank $r+1$ sublattice of $\mathrm{Pic}(S)$ which
contains a quasi-polarization must have signature $(1,r)$
by the Hodge index theorem.}}, $P_{\Delta,\delta}$ is empty unless $\Delta > 0.$  By definition, $P_{\Delta,\delta}$ is a reduced subscheme.

Let $h, d_{1}, \dots, d_{r}$ determine a positive discriminant
$$\Delta(h,d_{1},\dots,d_{r}) > 0.$$  The Noether-Lefschetz divisor
$D_{h, (d_{1},\dots,d_{r})}\subset \mathcal{M}_{\Lambda}$ is defined by 
the weighted sum
\begin{equation}\label{greff}
D_{h,(d_{1},\dots,d_{r})} 
= \sum_{\Delta,\delta} m(h,d_1,\dots,d_r|\Delta,\delta)\cdot[P_{\Delta,\delta}]
\end{equation}
where the multiplicity $m(h,d_1,\dots,d_r|\Delta,\delta)$ is the number
of elements $\beta$ of the lattice $(\mathbb{L},\iota)$ of 
type $(\Delta,\delta)$ satisfying
\begin{equation*}
\langle \beta, \beta \rangle = 2h-2,\ \  \langle \beta, v_{i}\rangle = d_{i}.
\end{equation*}
If the multiplicity is nonzero, 
then $\Delta | \Delta(h,d_{1},\dots,d_{r})$ so only 
finitely many divisors appear in the 
above sum.

\subsubsection{Refined Noether-Lefschetz numbers}
Let $X$ be a nonsingular projective Calabi-Yau 3-fold
fibered in $K3$ surfaces,
$$\pi: X \rightarrow \proj^1\ .$$
Let $L_1,\ldots,L_r\in \text{Pic}(X)$ determine 
a 1-parameter family of $\Lambda$-polarized $K3$ surfaces,
$$(X,L_1,\ldots,L_r,\pi)\ .$$
 The 1-parameter family determines a morphism
$$\iota:\proj^1\rightarrow \mathcal{M}_\Lambda\ .$$

The Noether-Lefschetz number $\mathsf{NL}^\pi_{h,d_1,\dots,d_r}$ is defined \cite{gwnl,PT2}
by the following conditions:
\begin{enumerate}
\item[$\bullet$] if $\Delta(h,d_{1},\dots,d_{r}) < 0$, then 
$\mathsf{NL}^\pi_{h,d_1,\dots,d_r}=0$,
\item[$\bullet$] if $\Delta(h,d_{1},\dots,d_{r}) = 0$, then 
$\mathsf{NL}^\pi_{h,d_1,\dots,d_r}= -2$,
\item[$\bullet$] if $\Delta(h,d_{1},\dots,d_{r}) > 0$, the
Noether-Lefschetz number is defined by the classical
intersection
product
\begin{equation*}\label{def11}
\mathsf{NL}^\pi_{h,(d_1,\dots,d_r)} =\int_{\proj^1} \iota_\pi^*[D_{h,(d_1,\dots,d_r)}].
\end{equation*}
\end{enumerate}

Our refinements of $\mathsf{NL}^\pi_{h,(d_1,\dots,d_r)}$
% $\mathsf{RNL}^\pi_{h,d_1,\dots,d_r}$ 
will not be numbers, but
rather representations{\footnote{As before, We denote the
irreducible representations of $SU(2)\times SU(2)$ by pairs
$[j_L,j_R]$ where $j_L,j_R\in \frac{1}{2} \mathbb{Z}_{\geq 0}$.}}
of $SU(2)\times SU(2)$
lying in the space
%$$\mathsf{RNL}^\pi_{h,d_1,\dots,d_r} \ \in\  
$$\mathbb{Z}_{\geq 0}\, [0,0]
\oplus \mathbb{Z}_{\geq 0}\, [0,\frac{1}{2}]\ .$$
The first refinement is defined by
$$\mathsf{RNL}^{\pi,\circ}_{h,d_1,\dots,d_r}= \mathsf{NL}^\pi_{h,d_1,\dots,d_r}\, [0,0]$$
and carries no more data than the Noether-Lefschetz number.

The definition of the second refinement $\mathsf{RNL}^{\pi,\diamond}_{h,d_1,\dots,d_r}$
is more subtle. Again, we consider
three cases based upon the
discriminant:
\begin{enumerate}
\item[$\bullet$] if $\Delta(h,d_{1},\dots,d_{r}) < 0$, then 
$\mathsf{RNL}^{\pi,\diamond}_{h,d_1,\dots,d_r}=0$,
\item[$\bullet$] if $\Delta(h,d_{1},\dots,d_{r}) = 0$, then 
$\mathsf{RNL}^{\pi,\diamond}_{h,d_1,\dots,d_r}= [0,\frac{1}{2}]$.
\end{enumerate}

If $\Delta(h,d_{1},\dots,d_{r}) > 0$, 
we divide the effective sum \eqref{greff} defining 
$D_{h,(d_1,\dots,d_r)}$ into two parts
$$D_{h,(d_1,\dots,d_r)} = S_\iota + T_\iota$$
where $S_\iota$ is the sum of the divisors on the right side of
 \eqref{greff} not containing
$\iota(\proj^1)$ and $T_\iota$ is the sum of such divisors
containing $\iota(\proj^1)$.
The final case of the definition is:
 \begin{enumerate}
\item[$\bullet$] if $\Delta(h,d_{1},\dots,d_{r}) > 0$, then 
$$\mathsf{RNL}^{\pi,\diamond}_{h,d_1,\dots,d_r}= 
\int_{\proj^1} \iota_\pi^* S_\iota\ \cdot [0,0]
-\frac{1}{2}\int_{\proj^1} \iota_\pi^* T_\iota\ \cdot [0,\frac{1}{2}] \ .$$ 
\end{enumerate}

The motivation of the second 
refinement is to record the geometric components of the
Noether-Lefschetz locus over the base $\proj^1$. Such loci here are unions
of points and lines --- the points correspond to the 
representation $[0,0]$ and the lines to the representation
$[0,\frac{1}{2}]$.

\subsection{Refined Pairs/Noether-Lefschetz correspondence}
We predict a refined Pairs/Noether-Lefschetz correspondence 
which intertwines
three theories associated to   
the 1-parameter family $$\pi:X \rightarrow \proj^1$$
 of $\Lambda$-polarized $K3$ surfaces
with Calabi-Yau total space:
\begin{enumerate}
\item[(i)] the refined Noether-Lefschetz theory  of $\pi$,
\item[(ii)] the refined Gopakumar-Vafa invariants of $X$ in fiber classes,
\item[(iii)] the refined  invariants $\mathsf{R}^h_{j_L,j_R}$
of the $K3$ fibers.
\end{enumerate}

Let $\mathsf{R}_{j_L,j_R,(d_1,\dots,d_r)}^X$ denote 
the refined Gopakumar-Vafa invariants of \cite{CKK} defined  via the stable
pairs moduli spaces of $X$
for $\pi$-vertical curve classes of degrees
%{\footnote{The
%invariant $n^X_{g,(d_1,\dots,d_r)}$ may be a (finite) sum of $n_{g,\gamma}^X$
%for $\pi$-vertical curve classes $\gamma \in H_2(X,\mathbb{Z})$.}} 
$d_1,\dots,d_r$
with respect to  line bundles $L_1,\dots, L_r$
corresponding to a basis of $\Lambda$. 
An $r$-tuple $(d_1,\ldots,d_r)$ is positive if the
associated degree with respect to a quasi-polarization 
$\lambda^\pi\in \Lambda$ is positive.

For our proposed R/NL refined correspondence, the 
$K3$ invariant $\mathsf{R}^h_{j_L,j_R}$ must be divided into two parts,
\begin{equation}\label{xdxd}
\mathsf{R}^h_{j_L,j_R}=
\mathsf{R}^{h,\circ}_{j_L,j_R} + 
\mathsf{R}^{h,\diamond}_{j_L,j_R}\ .
\end{equation}

\vspace{15pt}
\noindent {\bf Speculation [Refined P/NL correspondence].}
{\em A 1-parameter family of $\Lambda$-polarized $K3$ surfaces 
$$\pi: X \rightarrow \proj^1$$
with Calabi-Yau total space
 determines a division \eqref{xdxd} satisfying the following property.
For degrees $(d_1,\dots,d_r )$ positive with respect to the
quasi-polariza\-tion $\lambda^\pi$,
\begin{eqnarray*}
\sum_{j_L,j_R} \mathsf{N}_{j_L,j_R}^{X,(d_1,\ldots,d_r)}[j_L,j_R] & = & 
\sum_{j_L,j_R}
\sum_{h=0}^\infty 
\mathsf{R}_{j_L,j_R}^{h,\circ}[j_L,j_R] \otimes  
\mathsf{RNL}_{h,(d_1,\dots,d_r)}^{\pi,\circ}\ \\ & & +
\sum_{j_L,j_R}
\sum_{h=0}^\infty 
\mathsf{R}_{j_L,j_R}^{h,\diamond}[j_L,j_R] \otimes  
\mathsf{RNL}_{h,(d_1,\dots,d_r)}^{\pi,\diamond}\ .
\end{eqnarray*}}

\vspace{5pt}

By vanishing properties of the Noether-Lefschetz numbers, the above summations
over $h$ are finite for given $(d_1,\ldots, d_r)$.
We expect the counts
 $\mathsf{R}^X_{j_L,j_R,(d_1,\ldots,d_r)}$
to be invariant under deformations of 
$X$ {\em as a family of $\Lambda$-polarized $K3$ surfaces}.

\subsection{STU example}
For the STU model, we have a precise conjecture for the
division \eqref{xdxd} of $\mathsf{R}^h_{j_L,j_R}$ which is consistent
with several basic calculations.

We follow the STU conventions of \cite{KMPS} with the
lattice 
$$\Lambda=
\begin{pmatrix}
0 & 1 \\
1 & 0 
\end{pmatrix}\ .$$
The Noether-Lefschetz numbers for the STU model
$$\pi:X \rightarrow \proj^1$$
are determined in \cite{KMPS} to be
$$\mathsf{NL}^\pi_{h,(d_1,d_2)} = \text{Coeff}_{q^{1+d_1d_2-h}}
\Big(-2 E_4(q)E_6(q)\Big) \ $$
where $E_4$ and $E_6$ are the Eisenstein series,
$$-2 E_4(q)E_6(q) = -2 + 528 q+ 270864q^2 + 10393152q^3 + \ldots\ .$$ 
The refinement  is
easily seen to be given by
$$\mathsf{RNL}^{\pi,\diamond}_{h,(d_1,d_2)} = [0,\frac{1}{2}]$$
if $1+d_1d_2-h=0$  and
$$\mathsf{RNL}^{\pi,\diamond}_{h,(d_1,d_2)} = \mathsf{NL}^\pi_{h,(d_1,d_2)}\cdot 
[0,0]$$
otherwise.
The Betti number of $X$ are
$$u^{-3}[X]=  u^{-3} + 3 u^{-1} + 488 + 3 u + u^{3}\ .$$

We define $\mathsf{R}^{h,\diamond}_{j_L,j_R}$ for the STU model by a formula 
parallel to \eqref{DDD} but using only
part of the generating series of Hodge numbers
of the Hilbert schemes of points of $K3$ surfaces:
\begin{multline}\label{DDDD}
\sum_{h=0}^\infty \sum_{j_L}
\sum_{j_R} 
%(-1)^{2(j_L+j_R)}\mathsf{R}_{j_L,j_R}^h\ 
%[j_L]_{(-u)^2}[j_R]_{(-y)^2}
\mathsf{R}_{j_L,j_R}^{h,\diamond}\ 
[j_L]_{y}\, [j_R]_{u}
%\frac{(-u)^{-2j_L} - (-u)^{2j_L+2}}{1-u^2}
%\frac{(-y)^{-2j_R} - (-y)^{2j_R+2}}{1-q^2}
\, q^h= \\
\prod_{n=1}^\infty
 \frac{1}{
(1-u^{-1}y^{-1} q^n)(1-u^{-1}y q^n)(1-uy^{-1}q^n)
(1-u y q^n)} \,.
\end{multline}

\begin{table}[h]
\begin{center} {% \footnotesize 
\begin{tabular} {|c|c|} \hline ${{\mathsf{R}^{0,\diamond}_{\frac{i}{2},\frac{j}{2}}}_{}}$  & i=0 \\  \hline  j=0 & 1\\ \hline  \end{tabular}} 
\hspace{0.5cm}
\begin{tabular} {|c|cc|} \hline  ${\mathsf{R}^{1,\diamond}_{\frac{i}{2},\frac{j}{2}}}_{}$  & j=0 & 1 \\  
\hline  i=0 &   &  \\ 1  &  & 1 \\  \hline \end{tabular} 
\hspace{0.5cm}
\begin{tabular} {|c|ccc|} \hline 
 ${\mathsf{R}^{2,\diamond}_{\frac{i}{2},\frac{j}{2}}}_{}$   & j=0 & 1 & 2 \\  \hline i=0 & 1 &  &  \\  1 & 
 & 1 &  \\  2 &  &  & 1 \\  \hline \end{tabular} \hspace{0.0cm}
\end{center}
\end{table}
\vskip 0cm
\begin{table}[h]
\begin{center}
 \begin{tabular} {|c|cccc|} \hline  ${\mathsf{R}^{3,\diamond}_{\frac{i}{2},\frac{j}{2}}}_{}$  & j=0 & 1 & 2 & 3 \\  \hline i=0 & 1 &  & 1 &  \\  
1 &  & 2 &  &  \\  2 & 1 &  & 1 &  \\  3 &  &  & 
 & 1 \\  \hline \end{tabular} 
\hspace{0.5cm}
\begin{tabular} {|c|ccccc|} \hline  ${\mathsf{R}^{4,\diamond}_{\frac{i}{2},\frac{j}{2}}}_{}$ 
 &j=0 & 
1 & 2 & 3 & 4 \\  \hline i=0 & 3 &  & 1 &  &  \\  1 &  & 
3 &  & 1 &  \\  2 & 1 &  & 3 &  &  \\  3 &  & 1 &  
& 1 &  \\  4  &  &  &  &  & 1 \\  \hline 
\end{tabular}
% \begin{table}[h]
% \begin{center} 
% \begin{tabular} {|c|cccccc|} \hline $R^4_{\frac{i}{2},\frac{j}{2}}$
% & j=0 & 1 & 2 & 3 & 4 & 5 \\  \hline i=0 & 84777 
%&  & 253 &  &  &  \\  1 &  & 16171 &  & 22 &  &  \\  2 & 
% 253 &  & 2254 &  & 1 &  \\ 3 &  & 22 &  & 253 &  &  \\  
% 4 &  &  & 1 &  & 21 &  \\  5 &  &  &  &  &  & 1 \\  \hline \end{tabular} 
\vskip .4cm
\caption{All nonvanishing $\mathsf{R}^{h,\diamond}_{j_L,j_R}$ for $h\le 4$}
%\label{tableBettiK3}
\end{center}
\end{table}
\vskip -.7cm

%A complete list of nonvanishing
%$\mathsf{R}_{j_L,j_R}^{h,\diamond}$ for $h\leq 4$ is:
%$$\mathsf{R}^{0,\diamond}_{0,0}=1\, ,  $$
%$$\mathsf{R}^{1,\diamond}_{\frac{1}{2},\frac{1}{2}}=1\, , $$
%$$\mathsf{R}^{2,\diamond}_{0,0}=1\, , \ \
%\mathsf{R}^{2,\diamond}_{\frac{1}{2},\frac{1}{2}}=1\, , \ \
%\mathsf{R}^{2,\diamond}_{1,1} = 1\, .$$

%$$\mathsf{R}^{3,\diamond}_{0,0}=1\, , \ \
%\mathsf{R}^{3,\diamond}_{1,0}=1\, , \ \
%\mathsf{R}^{3,\diamond}_{\frac{1}{2},\frac{1}{2}}=2 \, , \ \
%\mathsf{R}^{3,\diamond}_{0,1}=1\, , \ \
%\mathsf{R}^{3,\diamond}_{1,1}=1\, , \ \
%\mathsf{R}^{3,\diamond}_{\frac{3}{2},\frac{3}{2}}=1\, .$$
%Finally, we present the nonvanishing 
%invariants for $h=4$:
%$$\mathsf{R}^{4,\diamond}_{0,0}=3\, , \ \
%\mathsf{R}^{4,\diamond}_{1,0}=1\, , \ \
%\mathsf{R}^{4,\diamond}_{\frac{1}{2},\frac{1}{2}}=3 \, , \ \
%\mathsf{R}^{4,\diamond}_{0,1}=1\, $$
%$$
%\mathsf{R}^{4,\diamond}_{\frac{1}{2},\frac{3}{2}}=1\, , \ \
%\mathsf{R}^{4,\diamond}_{1,1}=3\, , \ \
%\mathsf{R}^{4,\diamond}_{\frac{3}{2},\frac{1}{2}}=1\, , \ \
% $$ $$
%\mathsf{R}^{4,\diamond}_{\frac{3}{2},\frac{3}{2}}=1\, , \ \
%\mathsf{R}^{4,\diamond}_{2,2}=1\,  .$$

We expect $\mathsf{R}_{j_L,j_R}^{h,\diamond}$ to always be nonnegative
and bounded by $\mathsf{R}_{j_L,j_R}^{h}$.
Then,
$\mathsf{R}^{h,\circ}_{j_L,j_R}$ is uniquely defined
by equations \eqref{DDD}, \eqref{xdxd}, and \eqref{DDDD}.

\vspace{15pt}
\noindent {\bf Conjecture D.}
{\em  A refined $P/NL$ correspondence holds for 
fiber classes of the STU model:
\begin{eqnarray*}
\sum_{j_L,j_R} \mathsf{N}_{j_L,j_R}^{STU,(d_1,d_2)}[j_L,j_R] & = & 
\sum_{j_L,j_R}
\sum_{h=0}^\infty 
\mathsf{R}_{j_L,j_R}^{h,\circ}[j_L,j_R] \otimes  
\mathsf{RNL}_{h,(d_1,d_2)}^{\pi,\circ}\ \\ & & +
\sum_{j_L,j_R}
\sum_{h=0}^\infty 
\mathsf{R}_{j_L,j_R}^{h,\diamond}[j_L,j_R] \otimes  
\mathsf{RNL}_{h,(d_1,d_2)}^{\pi,\diamond}\ ,
\end{eqnarray*}
for degrees $(d_1,d_2)$ positive with respect to the
quasi-polariza\-tion.}
\vspace{10pt}

 Conjectures C and D together  predict
the refined invariants of the STU model in fiber classes. 
Let $(d_1,d_2)=(0,1)$ be the fiber class of the elliptic fibration
$$\mu: X \rightarrow \proj^1 \times \proj^1 \ .$$
For refined invariants of the STU model in class $(0,1)$,
the conjectures predict:
\begin{equation}  
\label{fddw} 488 { [0,0]}+ [\frac{1}{2},0] + [\frac{1}{2},1] \ .
%%\\[5mm]
%n=2: &281424[0,0] + 231[0,\frac{1}{2}] + 21[0,\frac{1}{2}] \\[2mm]
%&\ \ \  +528[\frac{1}{2},\frac{1}{2}] + 21[\frac{1}{2},1] 
% +[1,\frac{1}{2}]+ [1,\frac{3}{2}]
\end{equation}  
%{\bf \left(\frac{1}{2},\frac{1}{2}\right)} \\ [ 5 mm] 
%n=2: & 280962 {\bf \left(0,0\right)}+ 486  {\bf \left(\frac{1}{2},\frac{1}{2}\r%ight)} - 2 {\bf \left(1,1\right)}\\ [5 mm]
%n=3: & 15298438 {\bf \left(0,0\right)}+ 281448  {\bf \left(\frac{1}{2},\frac{1}%{2}\right)} + 486 {\bf \left(1,1\right)} \\ [3 mm]  &
%-2 {\bf \left(\frac{3}{2},\frac{3}{2}\right)} -2 ({\bf \left(1,0\right)} +{\bf %\left(0,1\right)}) \\ [5 mm]
%n=4: &  410133612 {\bf \left(0,0\right)}+   16209886 {\bf \left(\frac{1}{2},\frac{1}{2}\right)} +   281446 {\bf \left(1,1\right)}\\[3 mm] &+   486 {\bf \left(\frac{3}{2},\frac{3}{2}\right)} -2 {\bf \left(2,2\right)} + 
%  486({\bf \left(1,0\right)} +{\bf \left(0,1\right)})\\ [3 mm]& -2 
%\left({\bf \left(\frac{1}{2},\frac{3}{2}\right)} +{\bf \left(\frac{1}{2},\frac{%3}{2}\right)}\right) 
%\end{array}
%\end{equation*}

After expanding formula (8.1) of \cite{CKK} with the
refined invariants \eqref{fddw} for the STU model $X$
in class $(d_1,d_2)=(0,1)$, we obtain predictions
for the Betti realizations of the following stable pairs
motives:
\begin{eqnarray*}
u^{-2}[P_0(X,(0,1))] & = & u^{-2} + 2 + u^2\, , \\
{u^{-3}[P_1(X,(0,1))]} & = & u^{-3} + 3 u^{-1} + 488 + 3 u + u^{3}\ .
\end{eqnarray*}
The above predictions exactly match the expected geometry
\begin{eqnarray*}
P_0(X,(0,1)) & \cong&  \proj^1 \times \proj^1\, , \\
P_1(X,(0,1)) & \cong & X \ .
\end{eqnarray*}
In fact, the predictions for the fiber class $(0,1)$ case match for the
moduli spaces 
$P_m(X,(0,1))$ of stable pairs for all Euler characteristics $m$.

Conjecture D proposes an exact solution for the Betti realization of
the stable pairs motivic invariants for the STU model $X$ in
fiber classes.  Further values of the refined invariants
for the STU model are given below.

%\begin{eqnarray*}  
%(0,1): & 488 { [0,0]}+ [\frac{1}{2},0] + [\frac{1}{2},1] \\[5mm]
%(1,1): &280964[0,0] + [0,\frac{1}{2}] \\[2mm] & + [\frac{1}{2},0]+488[\frac{1}{%2},\frac{1}{2}] 
%  + [\frac{1}{2},1]\\[2mm] & + [1,\frac{1}{2}]+[1,\frac{3}{2}] \\[5mm]
%(2,1): &15928440[0,0] + 2[0,\frac{1}{2}] + [0,\frac{3}{2}] \\ [2mm]
%&2[\frac{1}{2},0]+281452[\frac{1}{2},\frac{1}{2}] 
%  +2[\frac{1}{2},1]\\[2mm]
% & + 2[1,\frac{1}{2}]+488[1,1] + [1,\frac{3}{2}] \\ [2mm]
%& + [\frac{3}{2},1] + [\frac{3}{2},2] \\[5mm]
%(3,1): &410133618[0,0] + 4[0,\frac{1}{2}] + 488[0,1]+[0,\frac{3}{2}] \\ [2mm]
%&3[\frac{1}{2},0]+16209892[\frac{1}{2},\frac{1}{2}] 
%  +4[\frac{1}{2},1] + [\frac{1}{2},2]\\[2mm]
% & + 488[1,0]+ 4[1,\frac{1}{2}]+281452[1,1] + 3[1,\frac{3}{2}] \\ [2mm]
%& + [\frac{3}{2},0]+ 2[\frac{3}{2},1]+488[\frac{3}{2},\frac{3}{2}] + 
%[\frac{3}{2},2] \\ [2mm]
%& + [2,\frac{3}{2}] + [2,\frac{5}{2}] 
%\end{eqnarray*}  

Checking the above prediction for $P_0(X,(n,1))$ is easy
for all $n\geq 0$.
Further checks in the case $(d_1,d_2)=(1,1)$ have been completed
(and match Conjecture D). 
Determining the
moduli space and the superpotential  becomes
harder as the Euler characteristic and the curve class increase.

\begin{table}[h]
\begin{center} 
%{% \footnotesize 
%\begin{tabular} {|c|cc|} \hline $R^0_{\frac{i}{2},\frac{j}{2}}$  
%      & i=0 & 1 \\  \hline 
%j=0 & & 1\\ \hline  \end{tabular}} 
\hspace{0.0cm}
\begin{tabular} {|c|ccc|} \hline  ${\mathsf{N}^{(0,1^{})}_{\frac{i}{2},\frac{j}{2}}}_{}^{}$  & j=0 & 1 &2  \\  
\hline  i=0 & 488 &  & \\ 1  &1 &  & 1 \\  \hline \end{tabular} 
\hspace{0.0cm}
\begin{tabular} {|c|cccc|} \hline 
 ${\mathsf{N}^{(1,1)}_{\frac{i}{2},\frac{j}{2}}}_{}$   & j=0 & 1 & 2 & 3 \\  \hline
 i=0 &280964  &1   & &  \\  
1 & 1 & 488 &1 &  \\  
2 &  & 1 &  &1  \\  \hline \end{tabular} \hspace{0.0cm}
%\end{center}
%\end{table}
\vskip .2 truecm
%\begin{table}[h]
%\begin{center}
\begin{tabular} {|c|ccccc|} \hline  
${\mathsf{N}^{(2,1)}_{\frac{i}{2},\frac{j}{2}}}_{}$ 
           & j=0 & 1               & 2          & 3  & 4  \\  \hline 
i=0     &15928440  &2           &        & 1  & \\  
1        & 2               & 281452 & 2    &    & \\  
2        &                  & 2           & 488& 1&  \\  
3        &                  &             &1      &    &1 \\  \hline \end{tabular} 
\vskip .2 truecm
\begin{tabular} {|c|cccccc|} \hline  ${\mathsf{N}^{(3,1)}_{\frac{i}{2},\frac{j}{2}}}_{}$  
&j=0 & 1 & 2 & 3 & 4 & 5 \\  \hline 
i=0 &410133618  & 4 & 488 & 1 & & \\  
1 & 3 & 16209892 & 4 &  & 1 & \\  
2 &  488 & 4 & 281452 & 3 & & \\   
3 & 1 &  & 2 & 488 & 1&  \\ 
4 &   &  &   & 1 & & 1 \\ 
\hline \end{tabular}
% \end{center}
% \end{table}
% \begin{table}[h]
% \begin{center} 
% \begin{tabular} {|c|cccccc|} \hline $R^4_{\frac{i}{2},\frac{j}{2}}$
% & j=0 & 1 & 2 & 3 & 4 & 5 \\  \hline i=0 & 84777 
%&  & 253 &  &  &  \\  1 &  & 16171 &  & 22 &  &  \\  2 & 
% 253 &  & 2254 &  & 1 &  \\ 3 &  & 22 &  & 253 &  &  \\  
% 4 &  &  & 1 &  & 21 &  \\  5 &  &  &  &  &  & 1 \\  \hline \end{tabular} 
\vspace{.5 cm}
\caption{Refined invariants for the STU model in fiber classes}
\label{tableSTU}
\end{center}
\end{table}

%\section{Local Calabi-Yau manifolds} 
%\label{localmodels}

%One obvious possibility to decouple the complex moduli is to 
%look at local models. In particular for all del Pezzo surfaces 
%$B$ embedded in a Calabi-Yau three manifold one can take a local 
%limit in which the local non-compact Calabi-Yau is described 
%by the total space of the canonical line bundle 
%${\cal O}(K_B)\rightarrow B$. We start with the del Pezzo's 
%which are elliptic fibrations over $\mathbb{P}^1$  

\section{Summary}

Our goal in the paper is to present a conjectural framework for
the evaluation of the motivic stable pairs invariants of 
 $K3$ surfaces in all curve classes.
The first step is a definition using suitably Noether-Lefschetz
transverse algebraic families. Conjectures A and B predict
a deformation invariance  for the associated Betti realization.
These new Betti properties go beyond the older numerical invariance.
Conjecture C then reduces the entire theory to the primitive (and
irreducible)
case which is determined by the Kawai-Yoshioka calculation.

In Conjecture D, we propose a refined Pairs/Noether-Lefschetz
correspondence for the STU model $X$. Together,
Conjectures A-D provide a rich non-toric setting where 
motivic invariants are well-behaved.

For Calabi-Yau 3-folds, there are three basic approaches to
curve counting (along with several variations):
\begin{enumerate}
\item[(i)] the Gromov-Witten theory of stable maps,
\item[(ii)]the Donaldson-Thomas theory of stable pairs, 
\item[(iii)] the
Gopakumar-Vafa theory of stable sheaves with $\chi=1$.
\end{enumerate}
While (i) and (ii) are foundationally well settled
and proven equivalent in many cases, the
precise formulation of (iii) is not yet clear
(and the equivalence is not understood).
As we have already mentioned, for a motivic theory, only 
the sheaf theoretic approaches (ii) and (iii) are
at present possible.

When the moduli space of stable sheaves with 1-dimensional support and
Euler characteristic 1 is {\em nonsingular}, approach (iii)
yields a clear proposal. Nonsingularity of the moduli spaces
is certainly rare.
However, in the Appendix
by R. Thomas,  nonsingularity is proven for the
stable sheaves associated to $K3$ surfaces in algebraic families
transverse to the relevant Noether-Lefschetz divisors. Nonsingularity
is proven for all (even imprimitive) curve classes.
The outcome provides a direct approach to the GV motivic
theory of $K3$ surfaces which matches precisely with our stable
pair predictions. In particular, the GV approach predicts
the divisibility invariance of the original KKV formula.

A path proving Conjectures A, B, and C is to start with the
GV moduli spaces as discussed in the Appendix and wall-cross
to stable pairs.{\footnote{See \cite{Toda} where exactly
the same strategy is applied to study the Euler characteristics of
moduli spaces of stable pairs on $S\times \com$. The difficulties
left open \cite{Toda} arise here as well.}}
 Various difficulties aries: the most signficant
of which may be the inability to achieve Euler characteristic
1 by twisting by line bundles in the imprimitive case. 
Nevertheless, proving a Pairs/GV correspondence appears
the most promising approach to Conjectures A, B, and C at the
moment.

\pagebreak 

\appendix
\section{Refined KKV from refined Gopakumar-Vafa}

\begin{center}{{\em by} R. P. Thomas}
\end{center}
\vspace{10pt}

%\noindent\textbf{Acknowledgements.} Grateful thanks are due to Arend Bayer for generous assistance in explaining the results of \cite{BM} in helping with \eqref{Beau} below. \medskip

Gopakumar and Vafa have suggested a conjectural approach to 
defining BPS numbers via moduli of sheaves. S. Katz \cite{SKatz}
proposed using stable sheaves of Euler characteristic $1$.
The papers \cite{HST,KL} make BPS and motivic predictions
of GV invariants via the geometry of such stable sheaves.
We show here
for fibre classes of sufficiently
Noether-Lefschetz transverse $K3$-fibred 3-folds, 
the proposals of \cite{HST,KL} work
 perfectly for both the BPS numbers and their refinements: 
we recover the generating series of Hodge numbers of Hilbert schemes 
of points on $K3$ surfaces \eqref{DDD} as considered in the current paper.

\medskip

By a flat family of stable sheaves on the fibres of a projective family
$$\pi\colon\Y\to B\, ,$$
we mean a coherent sheaf $\F$ on $\Y$ which is flat over $B$, such that $\F_b$ is stable on $\Y_b$ for any closed point $b\in B$. We will need the following standard result. 

\begin{lem} \label{piiso}
Let $\F$ be a flat family of stable sheaves on the fibres of $\pi\colon\Y\to B$. Then $\pi_*\hom(\F,\F)=\O_B$.
\end{lem}

\begin{proof}
Stable sheaves are simple, so
$$\Hom(\mathcal F_b,\mathcal F_b) \, \big/  \,   \C\cdot\id\ =\ 0
$$
for all closed point $b\in B$. It follows by base change that 
$$
\pi_*\hom(\mathcal F,\mathcal F)\, \big/\, \O_B\cdot\id\,=\ 0,
$$
which gives the result.
\end{proof}

Now fix $T\rt{\pi}(\Delta,0)$, a $K3$-fibred 3-fold with central fibre 
$\pi^{-1}(0)$ given by $$\iota\colon S\into T\, .$$ Suppose the curve $\Delta$ is transverse to $NL_\beta$ for some fixed
$$
\beta\in H^2(S,\Z)\cap H^{1,1}(S).
$$
By shrinking $\Delta$ if necessary, we may assume that $\pi$ is smooth and intersects $NL_\beta$ only in $0$.

Consider the moduli space $\MM_{1,\iota_*\beta}(T)$ of stable{\footnote{Stability is taken
with respect to a fixed polarization of $T$.}} 
%$$
%\MM=\MM_{1,\iota_*\beta}(T)
%$$
dimension 1 sheaves on $T$ with fundamental class $\iota_*\beta$ and holomorphic Euler characteristic $\chi=1$. The last condition ensures that no semistable sheaves exist, so $\MM_{1,\iota_*\beta}(T)$ is projective.
By simplicity, any stable sheaf must be scheme theoretically
supported on a single fibre of $\pi$. 
By the Noether-Lefschetz condition, there is an open and closed
component
$$\MM^\star \subset \MM_{1,\iota_*\beta}(T)$$
of sheaves with scheme theoretic support on $S$.
Hence, the obvious map
\beq{pushf}
\iota_*\colon\MM_{1,\beta}(S)\Into\MM^\star
\eeq
is a set-theoretic bijection.
%\footnote{There may be other components $\MM_{1,\gamma}(S)$ of $\MM$ far away
%with $\iota_*\gamma=\iota_*\beta$ .}

\begin{prop} \label{gxgx2}
The map \eqref{pushf} is an isomorphism of schemes.
\end{prop}

\begin{proof}
The statement is local, so we may work on a Zariski open subset $U\subset\MM$. Shrinking $U$ if necessary, we may assume there is a universal sheaf $\F$ on $U\times T$. We must show that $\F$ is the push forward of a sheaf on $U\times S$. The classifying map of the resulting flat family of stable sheaves on $S$ will define the inverse map to \eqref{pushf}.

Consider the composition
$$
\Gamma(\O_{\Delta\times U})\Rt{p^*}\Gamma(\O_{T\times U})\Rt\id\Hom(\mathcal F,\mathcal F)\ \cong\ \Gamma(\O_U),
$$
where the final isomorphism is given by Lemma \ref{piiso}. The map is
$\Gamma(\O_U)$-linear so defines a $U$-point of $\Delta\times U$. Its ideal is the kernel of the above composition, and by its definition this ideal annihilates $\mathcal F$.

Thus we get a map $$f\colon U\To \Delta$$ such that $\mathcal F$ is the pushforward of a sheaf $F$ on $T\times_\Delta U\,\subset T\times U$. Since $\F$ is flat over $U$, so is $F$.

On closed points $f$ is the constant map to $0\in \Delta$. To prove $f$ is
a constant map of schemes, we may replace $U$ by the formal neighbourhood of any closed point. The result will follow below from Hodge theory and the
Noether-Lefschetz transversality condition.

We have shown that $F$ is a flat family of stable sheaves on the smooth family $$T\times_\Delta U\Rt\pi U$$ of $K3$ surfaces over $U$. The fundamental class 
of $F$ lies in
$$
F^1H^2_{dR}(T\times_\Delta U/U),
$$
the part of the Hodge filtration defined by $\Omega^{\ge1}_{T\times_\Delta U/U}$.

Via the isomorphism \cite[Proposition 3.8]{Blo}
$$
H^2_{dR}(T\times_\Delta U/U)\ \cong\ H^2(S,\C)\otimes_\Delta \O_U,
$$
the class of $F$ is $\beta\otimes1$. Since it lies in $F^1$, its projection
$$
[\beta\otimes1]^{0,2}\in H^2_{dR}\big(T\times_\Delta U/U\big)\, \Big/\, 
F^1H^2_{dR}\big(T\times_\Delta U/U\big)
$$
vanishes identically. This class is the pull back via $f$ of the analogous class
$$
[\beta\otimes1]^{0,2}\in H^2_{dR}\big(\widehat T/\widehat \Delta\big)\, 
\Big/ \, F^1H^2_{dR}\big(\widehat T/\widehat \Delta\big),
$$
where $\widehat \Delta$ is the completion of $\Delta$ at $0$, 
and $\widehat \Delta$ is the completion of $T$ along $S$.

But the scheme theoretic Noether-Lefschetz locus is defined precisely by the vanishing of $[\beta\otimes1]^{0,2}$, and we assumed this is $\{0\}\subset \Delta$. Therefore $f|_\Delta$ is the constant map to $0\in \Delta$.
\end{proof}

By Proposition \ref{gxgx2},  $\MM^\star$ is a moduli space of stable sheaves on $S$. By \cite[Corollary 3.5]{Y},  $\MM^\star$ is nonsingular, nonempty, and deformation equivalent to $\Hilb^g(S)$, where $2g-2=\beta^2$. In particular, 
$\MM^\star$ has canonical orientation data.

We expect the Chow support map to the complete linear system in class $\beta$,
\begin{equation}\label{Beau}
\MM^\star\To|\O(\beta)|,
\end{equation}
to factor through a Lagrangian fibration onto a projective space
\begin{equation}\label{Pg}
\PP^g\,\subset\ |\O(\beta)|.
\end{equation}
The proof in most cases is well-known:

\begin{itemize}
\item If $\beta^2>0$ and $\beta$ is nef, then \eqref{Beau} is the classical {\em Beauville integrable system}, and the image $\PP^g$ is the whole linear system $|\O(\beta)|$. The generic element $C$ of $|\O(\beta)|$ is a 
nonsingular irreducible curve of genus 
$$g=1+\beta^2/2\, ,$$ and the fibre of \eqref{Beau} over $C$
 is the nonsingular Lagrangian torus $\Pic_gC\subset \MM^\star$.

\item
If $\beta^2=0$ and $\beta$ is nef  of divisibility $m>0$, 
then $\beta/m$ is the class of a fibre in an elliptic fibration $S\to\PP^1$.  
The sheaves parameterized by $\MM^\star$ are (the pushforward to $S$ of) 
rank $m$ sheaves on a single fibre, and the fibration \eqref{Beau} maps 
this sheaf to its support in $\PP^1=|\O(\beta/m)|$. In turn this $\PP^1$ embeds in $\PP^m=|\O(\beta)|$ as in \eqref{Pg} by the $m$th Veronese embedding.

%\item
%When $\beta^2\ge0$ but $\beta$ is not nef, we proceed as in \cite[Section 11]{BM}. A series of spherical twists in the structure sheaves of $-2$-curves makes $\beta$ nef, and takes the Gieseker stable sheaves of $\M$ to Bridgeland stable complexes with nef curve class. Wall crossing back to the Giesker stability chamber of the space of Bridgeland stability conditions, the moduli space undergoes only a finite sequence of birational transformations. The upshot is that $\M$ is birational to a moduli space of Gieseker stable sheaves in a nef class of the same square.
%Since the wall-crossing transformations leave the support of a complex unaltered, the birational map preserves the fibres of the Lagrangian fibration, which is therefore is preserved.

\item
If $\beta^2=-2$, $\MM^\star$ is a single point, and the claim
is trivial. 

\item
If $\beta^2 <-2$, $\MM^\star$ is empty.
\end{itemize}

The remaining case where $\beta^2\geq 0$ with $\beta$ not nef is likely to
follow from the above cases after using wall-crossing and the methods of 
\cite[Section 11]{BM}. Alternatively, it is not hard to show directly that the fibres of the map \eqref{Beau} are Lagrangian; that the image should then be $\PP^g$ is a standard conjecture (proved by Hwang when the image is nonsingular).
We leave the matter open here.

Finally, for the
cases treated above, we can follow the calculation of \cite[Section 8.3]{KL} (which in turn follows \cite{HST}) to determine the refined Gopakumar-Vafa 
invariants using the perverse Leray filtration for the fibration.\footnote{The cited references \cite{HST, KL} use the fibration $\Hilb^g(S)\to\PP^g$ induced by an elliptic fibration $S\to\PP^1$. This is a deformation of a Beauville-Mukai system (relative compactified Pic$_g$) of any complete linear system of genus $g$ in a primitive curve class on a K3 surface. (For instance if the surface is an elliptic fibration with a section, we may take the curve class to be a section plus $g$ fibres.) Thus, when our curve class $\beta$ is \emph{primitive}, our fibration $\MM^\star\to\PP^g$ can be deformed to $\Hilb^g(S)\to\PP^g$ to deduce that the Hodge numbers associated to the perverse Leray filtration agree with those in \cite{HST, KL}.

We are grateful to Davesh Maulik and Junliang Shen for pointing out that for multiple curve classes, there is no such deformation \emph{through Lagrangian fibrations}, due to results of Markman \cite[Theorem 1.5 and Proposition 1.7]{Ma}. Fortunately however, Shen and Yin have recently proved the beautiful result that the perverse Leray Hodge numbers are equal to the ordinary Hodge numbers for any smooth holomorphic symplectic variety with a Lagrangian fibration \cite[Theorem 0.2]{SY}. Since $\MM^\star$ and $\Hilb^g(S)$ are deformation equivalent (even if not as Lagrangian fibrations) their ordinary Hodge numbers agree. Thus the perverse Hodge numbers of $\MM^\star\to\PP^g$ are the same as those of $\Hilb^g(S)\to\PP^g$ used in \cite{HST,KL}.} At the level of Poincar\'e polynomials, the generating function\footnote{See the equation before \cite[Equation (8.3)]{KL}, before the specialization  $t_R=-1$.} obtained is:
$$
\prod_{m\ge1}\frac1{(1-t_Lt_Rq^m)(1-t_L^{-1}t_Rq^m)(1-t_Lt_R^{-1}q^m)(1-t_L^{-1}t_R^{-1}q^m)(1-q^m)^{20}}\,.
$$
Substituting $(u,v)$ for $(t_L,t_R)$ gives the refined KKV generating function 
\eqref{DDD} of KKP.

\vspace{10pt}
\noindent{\em Acknowledgements.} Grateful thanks are due to A. Bayer for generous assistance in explaining the results of \cite{BM} and helping with \eqref{Beau}. 
\noindent Department of Mathematics\\
University of Illinois\\
katz@math.uiuc.edu

\vspace{+8 pt}
\noindent Physikalisches Institut\\
Universit\"at Bonn\\
aklemm@physik.uni-bonn.de

\vspace{+8 pt}
\noindent
Departement Mathematik\\
ETH Z\"urich\\
rahul@math.ethz.ch

\end{document}